\documentclass{article}

\usepackage[left=2cm,right=4cm,top=1.5cm,bottom=1.5cm]{geometry}
\usepackage[utf8]{inputenc}
\usepackage{amsthm}
\newtheorem{definition}{Definition}
\usepackage{svg}
\usepackage{soul}
\usepackage{amssymb, bm}
\usepackage{parskip}
\usepackage{eurosym}
\usepackage{multirow}
\usepackage{multicol}
\usepackage{amsfonts}
\usepackage{paralist}
\usepackage{comment}
\usepackage{booktabs}
\usepackage{siunitx}
\usepackage{xcolor}

\usepackage{color, colortbl}
\definecolor{Gray}{gray}{0.9}
\definecolor{Reddish}{rgb}{0.9,0,0}
\usepackage[first=0,last=9]{lcg}

\usepackage[ruled,vlined]{algorithm2e}
\usepackage{float}
\usepackage{authblk}
\usepackage{algorithmic}
\usepackage[shortlabels]{enumitem}
\usepackage{lineno}
\usepackage{amstext} 
\usepackage{array}   
\newcolumntype{L}{>{$}l<{$}} 

\usepackage{caption}
\usepackage{subcaption}
\usepackage{amsmath}
\usepackage{enumitem}

\usepackage{tikz}
\usepackage[draft]{todonotes}
\usepackage{amsopn}

\usepackage{graphicx}
\usepackage{pgfplots}

\usepackage{cleveref}

\usetikzlibrary{arrows.meta}
\newtheorem{theorem}{Theorem}[section]
\newtheorem{proposition}{Proposition}[section]

\newtheorem{assumption}[theorem]{Assumption}
\newtheorem{remark}[theorem]{Remark}
\crefname{assumption}{Assumption}{Assumptions}

\usepackage{cleveref}
\usepackage[
backend=bibtex,
style=nature,
maxcitenames=2,
uniquelist=false,
uniquename=init,
citestyle=authoryear]{biblatex} 
\definecolor{rev}{rgb}{0.0, 0.0, 0}

\pgfplotsset{compat=1.18}

\newcommand{\Eqset}[0]{\mathcal{E}}
\newcommand{\profit}[0]{\Pi}
\newcommand{\feasibleset}[0]{\mathcal{X}}
\newcommand{\minwage}[0]{J_{\min}}

\title{Idle wage as a tool to regulate the relationship between ride-hailing platforms and drivers}

\author[1,*]{Andres Fielbaum}
\author[2,3]{David Salas}
\author[1]{Ruilin Zhang}
\author[4]{Francisco Castro}
\affil[1]{\small School of Civil Engineering, University of Sydney}
\affil[2]{Instituto de Ciencias de la Ingenier\'ia, Universidad de O’Higgins}
\affil[3]{Centro de Modelamiento Matem\'atico, CNRS IRL 2807, Universidad de Chile}
\affil[4]{Anderson School of Management, University of California Los Angeles (UCLA)}
\affil[*]{\small Corresponding author: andres.fielbaum@sydney.edu.au}
\date{}

\begin{document}
\maketitle
\begin{abstract}
    Ride-hailing platforms typically classify drivers as either employees or independent contractors. These classifications tend to emphasize either wage certainty or flexibility, but rarely both. We study an alternative or complementary approach: the \textit{Idle wage,} which provides with a fixed payment drivers even when they are connected without passengers on board.

 We adapt a well-known economic model of the supply-demand equilibrium in ride-hailing platforms and analyse how the optimal welfare and profit change with the introduction of an idle wage when drivers are risk-averse.
    We show that in a single-period setting, risk aversion implies that it is optimal to pay the drivers only through an idle wage. However, if the pool of available drivers is large, even a small idle wage could attract too many drivers, rendering the system unprofitable.
    When the demand fluctuates throughout multiple periods, we show that if the idle wage has to be constant, then it is optimal to combine the idle wage with the traditional payment via trips, so that surge pricing influences the number of drivers connected. This illustrates a relevant trade-off: Risk-aversion favours using the idle wage to attract drivers; however, if the platform is not allowed to fully adjust the idle wage over time, there may be periods in which the idle wage cannot resolve the mismatch between supply and demand. We propose a partially flexible constraint that still makes the idle wage-only solution viable. It allows the idle wage to adapt per period, as long as it fulfills a total minimum wage requirement.
    Numerical simulations suggest that the idle wage policy, if properly implemented, could be beneficial for the system as a whole.

\end{abstract}

\textbf{Keywords: Ride-hailing, Idle wage, Drivers, Riders, Equilibrium}


\section{Introduction}
The transport sector has been one of the most influenced by the so-called \textit{sharing economy}. The idea that online platforms connect two parties, one of them requesting a service and the other one offering it, has one of its most distinctive examples in ride-hailing companies, where passengers are matched online with drivers in order to complete their desired trips. This has been a major change compared to traditional taxis, as online coordination has proven to be much more effective than street cruising (\cite{cramer2016disruptive}). As such, these services have rocketed in the last decade, reaching thousands of passengers and drivers in numerous cities worldwide.

Ride-hailing services have motivated different research questions and regulatory challenges. One of the most controversial ones refers to the concept of the sharing economy itself: Are the drivers truly independent agents, or is it more appropriate to describe them as employees of the ride-hailing companies? The answer is nuanced. Drivers have significantly more freedom than in a traditional job, e.g. by selecting when to work, but other crucial aspects such as the price of each trip---and the subsequent driver's earnings---or weather they are matched with a rider are defined by the company.

From a policy perspective, this debate has been on the agenda of many countries and estates, which have taken different measures(\cite{fielbaum2023job}). On the one hand, the original situation in which the drivers are regarded as independent contractors has had several negative consequences, such as extremely long workdays, safety hazards, insufficient transparency on their income, and lack of social security (\cite{fielbaum2021sharing,fielbaum2023job,goletz2021ride,ashkrof2020understanding}). On the other hand, a strict regulation might worsen the situation: for example in Spain, after a new law was introduced, some delivery companies left the country, while others required riders to bargain, which in total seems to be leading to reduced earnings\footnote{See wired.co.uk/article/spain-gig-economy-deliveroo, accessed on 16/05/2024.}. All of this suggests that some alternative measures are required.

In this context, a novel alternative was proposed in Chile when discussing a bill that would regulate ride-hailing, which we translate as\footnote{See https://repositorio.uahurtado.cl/handle/11242/26023, accesed on 18/09/2023. Finally, a different bill was approved in May 2023 without including the Idle wage.} \textbf{Idle wage}. The basic idea is to \textit{provide a structure of payment to the drivers that valorizes the time when they are connected but not transporting any passenger}. This idea was borrowed from the general labour legislation, where it deals with situations in which workers are available to work, and the only reason why they are not performing any task is because the employer has not assigned them one. In the case of ride-hailing, there are at least two reasons that make this idea worth exploring:
\begin{enumerate}
    \item The situation in which drivers are idling has some similarity to the general case just described. When a driver is connected, they cannot work in anything else, so every minute spent without passengers is actually part of their daytime without a wage.
    \item The ride-hailing company benefits when more drivers are connected. The greater the number of idling drivers, the shorter the waiting times (\cite{castillo2022matching}), and thus the better the quality of service. Arguable, these benefits provided by idle drivers should not be captured only by the company but shared with the drivers. This network effect makes this discussion specific to the transport sector and not directly transferable to other areas facing similar challenges.
\end{enumerate}

In this paper, we formalise this idea and propose a stylised model to capture its effects. Crucially, we assume drivers to be risk-avert, meaning that they prefer knowing in advance how much they will earn. On the other hand, the system can control better the relationship between supply and demand when the drivers' earnings are linked to prices. The former argument supports the introduction of the Idle wage, whereas the latter suggests that the flexibility of the current scheme should be preferred. This trade-off appears clearly in our model. 

As a first work exploring formally the idea of Idle wage, we consider a simplified model to study the overall implications of such a policy. In particular, introducing an Idle wage would affect the number of drivers deciding to connect, which in turn changes the conditions faced by the demand; that is to say, the economic equilibrium between demand and supply requires a detailed analysis. In this sense, our model is based on an average steady situation. We also assume that drivers accept all trips and they only decide whether to enter or not to the platform at a given period. These simplifying assumptions are common in the literature, and they follow the same spirit of other works like \textcite{castillo2022matching,yan2020dynamic}. The resulting model allows us to study the tension between the valorisation of idling drivers and the efficiency of the overall system.

We first consider a single period setting and show that an idle wage-only policy---as opposed to dynamic pricing---is optimal for the platform. Because drivers are risk averse, they prefer to be paid via the idle wage. In addition, the stable demand setting reduces the need for a dynamic policy and, therefore, the platform is indifferent between traditional pricing and the idle wage. We establish, however, that the effectiveness of the idle wage diminishes with the pool of potential drivers, possibly rendering the system unprofitable under high minimum wage requirements. In a multi-period setting, we demonstrate that imposing a constant idle wage throughout the periods fails to flexibly adapt to varying demand requirements and bring enough supply when needed. In turn, when a constant idle wage is in place, it must be complemented with dynamic pricing. Interestingly, we also show that in an intermediate flexible setting the platform can realise the full benefits of a fully flexible idle wage. In particular, we impose a minimum idle wage requirement that guarantees drivers a certain level of earning during a block of hours e.g., 8 hours times the minimum hourly wage. 
In all, we show under various scenarios and in numerical simulations that introducing this policy could benefit drivers and the company at the same time.

The paper is structured as follows: Section \ref{sec:LitRev} provides a summarised literature review, and identify the contributions of the present work. Section \ref{sec:model_single_period} establishes the general model we use for a single period, which is then readily extended to multiple periods. Section \ref{sec:analysis_single_period} contains the analysis of the single-period model and our main results. Section \ref{sec:multiperiod} provides some alternatives to model the Idle wage in the multi-period setting, including a proposal for minimum wage. Finally, Section \ref{sec:NumericalExperiments} illustrate our results with numerical experiments on a case of study extracted from \textcite{yan2020dynamic}. We close the article in Section \ref{sec:conclusions} with the conclusions and some future perspectives.

\section{Literature review}\label{sec:LitRev}
The literature about ride-hailing systems has rocketed in the last years, involving many different aspects such as optimal assignment methods, pricing strategies, relationship with public transport, coordination vs competition between platforms, among others. \textcite{tirachini2020ride,jin2018ridesourcing} provide a general review on the area. In this section, we focus on those papers where the relationship between drivers and the company is analysed. A main departing point is that we focus on formalising the idea of the Idle wage and we analyse its potential impacts on the economic equilibrium of the system.

\subsection{Conditions of the drivers}
A significant body of research has focused on measuring and modelling the conditions faced by drivers. Some studies have raised relevant concerns, such as the ones by \textcite{henao2019analysis,zoepf2018economics}, who found that a large percentage of drivers in the US earn less than the minimum wage, \textcite{fielbaum2021sharing}, who identified relevant safety hazards, and \textcite{mantymaki2019digital} who identify a concern stemming from the disparity of power between the platform and the drivers. On the positive side, drivers consistently value the flexibility of deciding when to work (\cite{fielbaum2021sharing,gloss2016designing,mantymaki2019digital,hall2018analysis}), a virtue that has been formally modelled and found to increase welfare by \textcite{chen2019value}. It is worth remarking that drivers' conditions are not only relevant per se, but they also impact the performance of the platform. Particularly, income satisfaction plays a very significant role in drivers' engagement with these jobs (\cite{chen2023income}).

The mentioned analyses, mostly based on surveys or interviews, actually find heterogeneity among the drivers, as some work part-time and some full-time, so the trade-off between flexibility and stability is valued differently. This categorisation has been formalised with data-based approaches that leverage information about drivers' actual workdays.  \textcite{ramezani2022empirical} identified three types of drivers in Chengdu, China:  (i) part-time drivers working flexible hours, (ii) part-time drivers working in the evenings, and (iii) full-time drivers. Similarly, \textcite{di2022analysis} also clusterised the drivers but mainly distinguishing urban and rural areas.

\subsection{Operational implications of the company-driver relationship}
The Independence that drivers have regarding when to connect, which passengers to accept, and how to route, might impact the overall efficiency of the system in many different ways. The potential misalignments between the company's and drivers' interests was studied by \textcite{afeche2018ride} in the context of matching and repositioning and by \textcite{besbes2021surge} in the context of pricing and repositioning. 
Similarly, \textcite{de2022evolution} focus on the underlying dynamics, recognising that drivers make decisions in different time frames, namely the long-term decision of registering, and a day-to-day decision of driving. The model by \textcite{castillo2022matching} and \textcite{yan2020dynamic}, which we use as a basis in our models, focuses on the equilibrium between supply and demand, showing that in some circumstances it is possible to have two equilibria, one of them the well-known \textit{wild goose chase.}

The main departing feature in our model is that the number of drivers in the system does not only react to expected earnings, but for the same expected earning, they prefer to receive it via Idle wage rather than through serving trips. This is because the earnings through serving trips have a random component, namely how lucky was a driver's day in terms of having little idling time between consecutive trips, whereas the Idle wage is not subject to any randomness.

Companies can encourage an alignment between their goals and drivers' behaviour by giving them non-mandatory suggestions or by sharing complementary information. The question about when drivers are more inclined to follow such suggestions is studied by \textcite{ashkrof2022ride}, with a focus on the acceptance of trips, and \textcite{guo2023seeking}, who concentrate on the cruising behaviour. The value of shared information was recently studied by \textcite{LiberonaEtAl2024Shared}, using a bilevel formulation for the allocation problem of idle drivers.

The full relationship between drivers and the platform is actually more complex, as drivers can coordinate to act strategically in order to induce certain responses. This has been examined by \textcite{schroder2020anomalous}, who show that drivers coordinate to create sudden supply shortages and thus induce surge pricing. \textcite{tripathy2022driver} proposes that a solution for this problem is a sort of carrot-and-stick approach, giving bonuses for drivers when they stay connected and/or a freezing time after disconnecting. 


\section{Model}\label{sec:model_single_period}
We now present a steady-state base model of a ride-hailing platform that provides rides to passengers by connecting them with drivers. The main stakeholders in our model are (i) the platform, which sets wages and prices and matches riders to drivers, (ii) the riders who, depending of the price and pickup time, decide whether or not to request a ride, and (iii) the drivers, who decide whether or not to drive for the platform in equilibrium. A salient component of our model is the {\it Idle wage}, through which drivers receive a constant income as long as they are available to serve rides. We note that in our base model, we consider a single-period case, which can be viewed as a setting where demand is stable over time. In \Cref{sec:multiperiod}, we extend our model to a multi-period setting. 

\begin{remark}\label{Remark:IdleForEveryone} It could be possible to model Idle wage as a payment only for idling drivers. However, this may lead to incorrect incentives depending on the balance between the Idle wage and the earnings per served trip. To avoid this issue, we propose an Idle wage that is paid regardless the status of the driver (idle, serving a trip, or in transit to serve a trip).  As we comment below, this formulation is equivalent in the single-period case to have Idle wage only for idle drivers, restricted to the case in which it is more convenient to be serving a passenger than to be idling.
\end{remark}

We now explain in detail the main components of our system: the demand, supply, the notion of system equilibrium, and the optimality problem faced by the platform.

\paragraph{Demand.} Let $p$ be the price set by the platform and $T$ be the average pickup time (which will eventually be determined in equilibrium) . We use $D(p,T)$ to denote the demand of riders who are willing to pay the price $p$ and are willing to wait $T$ until a drivers picks them up. We impose the following regularity conditions of the demand function. 

\begin{assumption}
$D:\mathbb{R}^2_+\to\mathbb{R}_+$ satisfies:
\begin{enumerate} 
    \item $D$ is bounded, continuously differentiable and it is strictly decreasing with respect to $p\geq 0$ and with respect to $T\in [0,\infty)$. 
    \item $D$ fulfills: 
    \begin{align}
T\cdot D(p,T) &\xrightarrow{T\to\infty} 0,&\forall p\geq 0.\label{eq:QuickEvanescent-T}\\
p\cdot D(p,T) &\xrightarrow{p\to\infty} 0, &\forall T\in[0,\infty).\label{eq:QuickEvanescent-p} 
\end{align}

\item We assume that the function $D(\cdot,0)$ is integrable, that is, $\int_0^{\infty}D(z,0)dz < \infty$, and that
    \begin{equation}\label{eq:EvanescentSurplus}
        \int_0^{\infty}D(z,T)dz \xrightarrow{T\to\infty}0.
    \end{equation}

\end{enumerate}
\end{assumption}
The first condition ensures that if the price or waiting time increases, fewer passengers will use the system. The assumption about $D$ being bounded is a natural condition as the size of the population is limited. 
The second condition establishes that as the average pickup time or the price increase, the demand for rides goes to zero. 
Note that in practice, if $p$ or $T$ are too large, the demand would be exactly zero. The assumptions given by Eqs. \eqref{eq:QuickEvanescent-T}-\eqref{eq:EvanescentSurplus} are weaker than assuming that $D$ becomes zero.

For a given pickup time $T$ and a given price $p$, the \textit{surplus} of the riders is computed as
\begin{equation}\label{eq:surplus}
    S(p,T) = \int_p^{\infty} D(z,T)dz,
\end{equation}
which represents all the extra utility of riders that are effectively taking the service. The surplus computes how much the active riders were willing to pay, on top the price $p$, for the fixed pickup time $T$. 

\paragraph{Supply.} In our model, drivers receive wages from two sources: (i) a percentage $1-\tau$, with $\tau \in (0,1)$, of the price $p$ for each served trip ($\tau$ represents the commision charged by the platform), and (ii) a constant Idle wage $J$ from being available to serve trips. We next describe in detail both sources of income and explain how they affect drivers' decisions. 

Let $Q$ be the total number of trips served in equilibrium per unit of time. For a price $p$, the total earning of drivers from trips is $(1-\tau)p Q$. In turn, denoting by $L$ the equilibrium total number of drivers on the platform, a given driver earns on average 
\begin{equation}\label{eq:expectedEarnings}
    e = (1-\tau)p\frac{Q}{L},
\end{equation}
per unit of time. We stress that in the expression above represents an average quantity, and drivers may not receive this exact income every period due to random variations. If $L = 0$ (and consequently, $Q=0$), we adopt the convention that $e=0$.

The total number of drivers $L$ on the platform is given by a supply function $\ell$ that captures the willingness to work of drivers and depends on the average earnings from trips $e$ and the Idle wage $J$, that is,
\begin{equation}\label{eq:supply_equilibrium}
L = \ell(e,J).
\end{equation}
We note that the Idle wage is received by any driver that is connected to the platform (including those drivers already serving trips). In contrast, the average earnings from trips $e$ is only earned by those drivers who are actively serving a trip. As discussed in Remark \ref{Remark:IdleForEveryone}, from a methodological perspective this is readily equivalent to pay the Idle wage only to idling drivers, and reduce the commision $\tau$ (which only matters while there are riders) accordingly. Therefore, for the rest of the paper we assume the idle wage is paid constantly while a driver is connected, regardless of their status. Moreover, from an implementation perspective, it might be more transparent to have a basic Idle wage all the time, and consider a served trip as a surplus\footnote{Note that an actual implementation of Idle wage would require some additional measures to ensure that nobody connects without driving, just to receive the Idle wage. In this simplified model, we assume that drivers always accept, or are mandate to accept, their rides.}.  

We consider the following assumptions on $\ell$:

\begin{assumption}\label{ass:supply} $\ell:\mathbb{R}^2_+\to\mathbb{R}_+$ satisfies:
    \begin{enumerate}
    \item $\ell$ is continuously differentiable at every point $(e,J)\in\mathbb{R}^2_+$.
    \item $\ell(0,0) = 0$, and $\frac{\partial \ell}{\partial e}, \frac{\partial \ell}{\partial J} >0$.
    \item For every point $(e,J)\in \mathbb{R}^2_+$, $\frac{\partial \ell}{\partial e}(e,J) \leq \frac{\partial \ell}{\partial J}(e,J)$.
\end{enumerate}
\end{assumption}
The first part is a technical condition that makes the model amenable to analysis.
 The second part states that the greater the income, the more drivers will be willing to work. The third assumption above is a crucial one and the main innovation from our model with respect to the existent literature (e.g., \cite{castillo2022matching}). This assumption implies that \textit{drivers prefer receiving one extra dollar via Idle wage rather than receiving it via trips.} This assumption relates to drivers' aversion to risk and it captures that in practice the Idle wage has no associated uncertainty, that is, under the Idle wage drivers will receive $J$ even if they don't serve trips. We differentiate between two possible degrees of risk aversion of drivers.  

 \begin{definition}[Risk-Aversion] We say that
 \end{definition}
\begin{enumerate}
    \item drivers are \textbf{risk-averse} if $\frac{\partial \ell}{\partial e}(e,J) < \frac{\partial \ell}{\partial J}(e,J)$, $\forall e,J$;
    \item and that drivers are \textbf{risk-neutral} (or that there is no aversion to risk) if  $\frac{\partial \ell}{\partial e}(e,J) = \frac{\partial \ell}{\partial J}(e,J)$, $\forall e,J$.
\end{enumerate}

Risk-averse drivers prefer to earn a deterministic wage than uncertain earnings. For risk-neutral driver, the impact of extra earnings is the same no matter if it comes from trips or from Idle wages.  In the risk-neutral case, for every $(e,J)\in \mathbb{R}^2_+$, we can easily deduce that
    \begin{equation}\label{eq:no-risk}
        \ell(e,J) = \ell(e+J,0) = \ell(0,e+J).
    \end{equation}

Finally, we define the \textit{social cost} of drivers which we will later use for welfare optimisation. The \textit{social cost} (or opportunity cost) of the drivers is the minimum income they would require to be willing to work, so that any additional income is welfare. Since Idle wage $J$ is a predictable salary, we consider that this opportunity cost is given by the value of $J$ at which drivers are willing to work, assuming $e=0$. Thus, the social cost of a  force $L$ is given by

\begin{equation}
    C(L) = \int\limits_0^L \ell_0^{-1}(z) dz,
\end{equation}
where $\ell_0(J) := \ell(0,J)$.

\paragraph{System equilibrium.} The system equilibrium is determined by the dynamics of the system and the decisions of drivers and riders. We first consider the dynamics of the system. When a driver is matched to a rider, the driver spends $T$ units of times picking up the passengers. However, this time fundamentally depends on the number of idle drivers $I$: if there are more idle drivers, then on average the distance between the driver and the rider will decrease.
We make the following assumptions.

\begin{assumption} $T:[0,\infty)\to (0,\infty]$ satisfies:
    \begin{enumerate}
    \item $T$ is strictly decreasing.
    \item $T(I)\xrightarrow{I\to \infty} 0$. 
    \item $T(0) = \infty$ and $T(I)\xrightarrow{I\to 0} \infty$.
    \item $T$  is continuously differentiable on $(0,\infty)$.
\end{enumerate}
\end{assumption}

The first three assumptions follow directly from the physical interpretation of the function $T(I)$. The fourth is needed to enable the equilibrium and optimality analysis. Observe that because $T(I)$ is strictly decreasing, then it is invertible, and that the function $I:(0,\infty]\to [0,\infty)$ defined by $I(T) = T^{-1}(T)$ is well defined, strictly decreasing and continuously differentiable on $(0,\infty)$.

After the driver has completed the pickup, the trip starts. We use $t$ to denote the length of an average trip. Hence, for every passenger in the system, a drivers spends an average of $t+T(I)$ units of time serving them. Given the system's dynamics, at equilibrium, the labour force in the system must satisfy the following balance equation:
\begin{equation}\label{eq:balanceLaborForce}
    L = I + (t + T(I))Q.
\end{equation}

\Cref{eq:balanceLaborForce} means that the total number of drivers $L$ on the platform can be computed as the sum of $I$ idle drivers plus $(t+T(I))Q$ drivers that are serving a rider. This can be interpreted as Little's Law in the context of ride-hailing (\cite{zhang2021pool}).

In addition, the demand must match the served trips, that is,
\begin{equation}\label{eq:DemandBalance}
    Q = D(p,T(I)).
\end{equation}
We are ready to introduce our equilibrium notion. 
We say that $(e,I,L,Q)$ forms an equilibrium for a fixed tuple $(p,J,\tau)$, if it belongs to the set
\begin{equation}\label{eq:EquilibriumSet}
    \Eqset(p,J,\tau) = \left\{ (e,I,L,Q)\ :\ 
    \begin{array}{l}
    \begin{array}{rcl}
       Q &=& D(p,T(I))\\
       L &=& I+ (t+T(I))Q\\
       e  &=& (1-\tau)p\frac{Q}{L}\\
       L &=& \ell(e,J)\\ 
    \end{array}\\
    e,I,Q,L\geq 0
    \end{array}\right\}
\end{equation}
That is, in equilibrium the number of passengers in the system equal the effective demand for trips, the total  force is consistent with the idle and busy drivers, and drivers who join the system do so taking into account their potential trip earnings and the Idle wage.

\paragraph{Platform's problem.} We will consider two cases: profit and welfare maximisation. Let $\bold{x} = (p,J,\tau,e,I,Q,L)$, denote the variables that determine the system. Recall that ride-hailing platform only controls directly the first three. The other four describe the equilibrium (or multiple equilibria) induced by $(p,J,\tau)$. 

The platform's profit is given by 
\begin{equation}\label{eq:profit}
\profit(\bold{x}) \triangleq \tau pQ - J\cdot L,
\end{equation}
where the first term is the revenue due to the total trips served (i.e., the commission kept by the platform), and the second corresponds to the Idle wage payment given to the drivers. 

The system's welfare is given by
\begin{equation}\label{eq:welfare}
W(\bold{x}) \triangleq S(p, T(I)) + pQ - C(L), 
\end{equation}
where the first term is the surplus of the passengers, and the second and third  terms together represent the joint surplus of the company and drivers. Note that the total Idle wage $JL$ doesn't appear in the formula of welfare, because the Idle wage can be understood as a welfare redistribution mechanism between the company and the drivers.

 The company aims to solve the following optimisation problem:
\begin{equation}\tag{$\mathcal{P}$}\label{eq:Problem-single-period}
    \begin{array}{rl}
        \displaystyle\max_{\bold{x}=(p,J,\tau,e,I,L,Q)} & \theta(\bold{x})  \\
         s.t & \begin{cases}
             p,J\geq 0,\\
             \tau \in [0,1],\\
             (e,I,L,Q)\in \Eqset(p,J,\tau),
         \end{cases} 
    \end{array}
\end{equation}
where $\theta(\bold{x})$ is either the profit of the company $\profit(\bold{x})$ or the welfare of the system $W(\bold{x})$. In what follows, we will also consider the optimal value (profit/welfare) with respect to a fixed commission $\tau\in [0,1]$. Thus, we define the function $v:[0,1]\to \mathbb{R}$ given by 
\begin{equation}\label{eq:ValueFunction}
    v(\bar{\tau}) = \max_{\bold{x}}\left\{ \theta(\bold{x}) \,:\ p,J\geq 0, (e,I,L,Q)\in \Eqset(p,J,\tau), \tau = \bar{\tau} \right\}.
\end{equation}
That is,  $v(\bar{\tau})$ as the maximum attainable $\theta$ (i.e., profit or welfare) for a given $\bar{\tau}$.
We remark that previous studies (e.g., \cite{castillo2022matching,yan2020dynamic}) have considered $\tau$ as exogenous, because this variable does not change as often as the others, that is, in short time scales the commission kept by the platform remains relatively constant. In our case, we consider the platform's problem at a more tactical level when is deciding on a payment structure that will likely stay unchanged in the short to medium term. At this more tactical level, the platform faces a natural trade-off between $1-\tau$ and $J$. Both are sources to pay drivers, with $J$ being paid with certainty to all drivers on the platform and the fraction $(1-\tau)$ of a fare only being passed to drivers who complete trips. Our aim is to shed-light on how the platform should go about balancing these two different ways of paying drivers. 

\section{Analysis and Efficiency of Idle Wage}\label{sec:analysis_single_period}
In this section, we first establish the existence of equilibria, and prove a quasi-uniquess property and show, in turn, that problem \eqref{eq:Problem-single-period} admits a solution. Then, we analyse the solution to \eqref{eq:Problem-single-period} and study the extent to which the Idle wage can help or hurt the platform's objective. 

\subsection{Existence of Economic Equilibria}

\begin{proposition}\label{thm:ExistenceEquil}
For any tuple $(p,J,\tau)\in \mathbb{R}^2_+\times [0,1]$, the set of equilibria $\Eqset(p,J,\tau)$ is nonempty.
\end{proposition}
\begin{proof} 
In the appendix.
\end{proof}

It is crucial to realise that  \Cref{thm:ExistenceEquil} ensures that there is \textit{at least} one equilibrium, but it does not discard the existence of multiple. In fact, having several equilibria is a frequent finding in ride-hailing markets, where some of them might be undesirable because the same throughput could be achieved in a more/less efficient way. The classical example here is the ``wild goose chase'': we refer the reader to the paper by \textcite{castillo2022matching} for a detailed description of how the existence of two equilibria can be identified with this phenomenon.

While it is possible to have multiple equilibria, there are some limits to how similar they can be. Namely, once the platform has decided on $p,J,\tau$, there is at most one equilibrium for each demand level. In other words, so far we have characterised the equilibria by the tuples $(e,I,L,Q)$; however, these variables are actually intertwined in a way that fixing $Q$ determines the value for all the others in an equilibrium. This is formalised in the following proposition.

\begin{proposition} \label{thm:OneVariableSuffices}
Let $(p,J,\tau)\in \mathbb{R}^2_+\times [0,1]$, and let two tuples $(e_i,I_i,L_i,Q_i)\in \Eqset(p,J,\tau)$ with $i=1,2$. If $Q_1=Q_2$, then both tuples are equal.
\end{proposition}
\begin{proof}
In the appendix.
\end{proof}

\begin{proposition} \label{Thm:ExistenceOfSolution}
Problem \eqref{eq:Problem-single-period} admits a solution.   
\end{proposition}
\begin{proof} 
In the appendix.
\end{proof}

Propositions \ref{thm:ExistenceEquil} and \ref{Thm:ExistenceOfSolution} ensure that the problem we have posed can indeed be analysed. First, the company can choose whatever values for $p,\tau,J$, and be sure that an equilibrium will be reached (\Cref{thm:ExistenceEquil}); second, there is at least one optimal value for each of these variables (\Cref{Thm:ExistenceOfSolution}). While these results might seem intuitive, the problem is far from simple because there is no convexity involved. 

\subsection{The Scope of the Idle Wage}

We now analyse what is the role played by $J$.
We first show that there is always an optimal solution with $\tau=1$, i.e., where the drivers only receive earnings via the Idle wage. However, we then establish that if the pool of potential drivers is large, the revenue maximising Idle wage will tend to stay low. That is, when demand is stable over time, the platform would prefer to pay drivers an Idle wage, the amount of which is limited by the number of potential drivers in the market.

To formalise our first result, recall that $v(\tau)$ is the maximum attainable $\theta$ (i.e., profit or welfare) for a given $\tau$, as given in \eqref{eq:ValueFunction}.

\begin{theorem} \label{thm:tau1isoptimal}
The function $v(\tau)$ is nondecreasing, and so, $v(\tau)$ attains its maximal value at $\tau = 1$. 
\end{theorem}
\begin{proof}
The core of the proof consists on: First, take a solution with $\tau<1$; second, take part of what is being paid to the drivers via $e$ (i.e., when taking passengers), and replace it by the same payment via $J$. As a result, the drivers will be better-off (due to risk aversion) and nobody will be worse-off. The formalisation of this argument is provided in the appendix.
\end{proof}

The main reason for this results is two-fold: risk-aversion and stable demand (single period setting). First, note that risk-aversion implies that drivers prefer to receive their income in a non-risky way, whereas the platform is indifferent about how to pay drivers. Furthermore, if the drivers were risk-neutral (Eq. \ref{eq:no-risk}), then any $\tau$ could be optimal because from the drivers' perspective both sources of payment are equivalent. Second, in our base model, demand doesn't change over time. This means that the platform's supply needs to remain constant and, therefore, the platform doesn't face a potential supply and demand mismatch due to too many drivers being online during low demand periods. As we will see in \Cref{sec:multiperiod}, \Cref{thm:tau1isoptimal} will no longer be valid in the multi-period case.


Our main result so far is that if the demand is stable in time, it is always optimal to pay the drivers only via Idle wage. This result is valid for both profit or welfare maximisation, and stems from drivers being risk averse. We next shed-light on the effectiveness of the Idle wage as a tool for providing a worthwhile wage for drivers.
In line with the most common practice in the ride-hailing industry, we will focus on the profit-maximisation case and study how the size of the pool of potential drivers impacts the magnitude of the $J$.

Let $A>0$ be the size of the pool of potential drivers so that 
\begin{equation} \label{Eq:EllAlternative}
    \ell(e,J)=A\cdot \rho(e,J),
\end{equation}
where $\rho(e,J)$ is a function taking values in $[0,1]$ satisfying \Cref{ass:supply}.
We will now show that if $A$ is too large, then the Idle wage might be a double-edge sword. Indeed, in contrast to the payment that drivers receive due to trips served (via $p$), $J$ has to be paid to all  drivers that are available to serve trip requests, regardless of the demand. Therefore, if the number of drivers becomes too large, 
the platform would not be able to profitably sustain a high Idle wage. 

For the purpose of the next result, let us use $J^*(A,\tau)$ to denote the optimal value of $J$ as a function of $A$ for a given $\tau$.\footnote{
$J^*$ also depends on all of the other variables and functions involved. Nevertheless, in this theorem it is only $A$ that changes, so this notation is sufficient.}

\begin{theorem} \label{Thm:JisZeroForLargeA}
    For all $\tau<1$, $J^*(A,\tau) \xrightarrow{A \rightarrow \infty} 0$.
\end{theorem}

\begin{proof}
The crux of the proof is that, for any given $J$ (no matter how small), if $A$ is large enough the company's profit would turn negative, which is worse than having $J=0$. This argument is formalised in the appendix.
\end{proof}

Assuming $\tau<1$ is needed as, otherwise, we know that $\tau=1$ is optimal and therefore $J$ can't be too low. However, this is not a strong assumption: it is the current practice, and moreover,  as we will see in section \ref{sec:multiperiod}, the optimal $\tau$ is indeed usually lower than 1 when we account for multiple periods. 

In the proof of \Cref{Thm:JisZeroForLargeA}, a novel component is that if $J$ is larger than the optimal, the profit could actually become  negative. This highlights a fundamental challenge associated to a policy like the Idle wage. \Cref{Thm:JisZeroForLargeA} suggests that, in the case of many drivers available, the platform would prefer to set $J$ the minimum allowed value.  It also establishes that any constraints on this minimum could lead to the platform opting not to offer the service at all.
Importantly, it is a well-known fact that, in some countries, driving in ride-hailing apps is especially attractive for some of the most vulnerable populations that need some easy way to get a job, e.g., recent migrants (\cite{de2023ridesourcing,james2022business}). One potential explanation of why any idea similar to Idle wage has not been implemented could be the fear of attracting too many drivers, and the need to pay that Idle wage to all of them regardless of the demand level.\footnote{A probably stronger reason is that most platforms intend to prevent laws defining the drivers as employees. A constant wage could be a strong argument to legally regard them as employees. However, this argument is of a completely different nature than the one derived from \cref{Thm:JisZeroForLargeA}, as the latter is directly related with reaching an efficient equilibrium, and therefore should be of strong interest to the regulator.} This is a relevant concern that would need to be assessed and eventually addressed for any potential policy akin to the Idle wage.

\section{Multi-Period Model} \label{sec:multiperiod}
We now consider a multi-period setting in which a day is divided into 
$h=1,\ldots,H$ periods of the same length. Each period has a different demand $D_h$ and labour supply $\ell_h$ functions. We use $\bold{D}$ and $\bold{\ell}$ to denote the vectors of demand and supply functions, respectively. 
The platform's problem in this setting is similar to \eqref{eq:Problem-single-period}, with the consideration that most quantities should be indexed by $h$, and we impose that the system must be in equilibrium in every period. 

The main difference with respect to \Cref{sec:model_single_period,sec:analysis_single_period} is in the Idle wage. As it is typically observed in ride-hailing platforms, we assume that the commission $\tau$ will remain constant throughout the day. However, we allow for the price $p$ and the Idle wage $J$ to vary over time. In particular, we allow the price to be fully flexible and we use $p_h$ to denote the price in period $h$ and $\bold{p}$ to denote the corresponding vector of prices. For the Idle wage, we consider three possible cases: (i) $J$ is fully flexibly (i.e., there is a different $J$ for every $h$); (ii) $J$ is constant throughout the day; and (iii) $J$ is allowed to vary throughout a day while satisfying certain operational constraints (see \Cref{sec:multiperiodintervals} for more details).

Case (i) is a simple extension of our single-period setting. Indeed, because in this case $J$ is fully flexible, the platform can set $\tau=1$ for each period and then the problem decouples so that the analysis and results from \Cref{sec:analysis_single_period} remain valid. Cases (ii) and (iii) require a more in-depth treatment, which we provide in \Cref{sec:SingleJ,sec:multiperiodintervals}. As we will see, as we limit the flexibility of $J$, a trade-off emerges. Since drivers are risk-averse, the platform would prefer to use $J$ to attract drivers. However, if the platform is not allowed to fully change $J$ over time, there may be periods in which an unflexible $J$ cannot handle the mismatch between supply and demand. Instead, $\tau$ multiplies the prices (which are fully flexible, as the observed surge pricing policies), and thereby affects both demand and supply, the platform may find it attractive to set $\tau < 1$ in an attempt to resolve any potential mismatches.

\subsection{The Limits of the Idle Wage and the Necessity of Dynamic Pricing}
\label{sec:SingleJ}
We know that in the single-period setting it is optimal for the platform to set $\tau=1$, i.e., to pay the drivers a fixed $J$ and keep all what is paid by the users. We now illustrate via a simple example how this result is no longer valid when the demand varies throughout the day, if $J$ cannot vary. 

We consider a two-periods scenario in which both $\tau$ and $J$ remain constant in both periods. Recall that in our base model formulation, $D$ is a function of $p$ and $T$, where $T$ depends on the number of drivers $L$ through $I$. For the example we introduce next, we consider a simplified representation, in which $D$ depends on $p$ and on $L$ directly.

\paragraph{Example.} 
Consider a profit-maximising platform, and assume that there is one high-demand and one low-demand period. We use the sub-indices $h,l$ to denote them. In the high-demand period, $D_h(p_h,L_h)=L_h$ if $p_h\leq 1$, and $0$ otherwise, whereas in the low-demand case $D_l(p_l,L_l)=L_l/2$ if $p_l\leq 1/2$ and $0$ otherwise. Obviously, the profit-maximising platform would set the prices to $p_h=1$ and $p_l=1/2$. 
Furthermore, we take $\ell(e,J)=e+(1+\varepsilon) J$ with $\varepsilon>0$ for both periods, meaning that drivers value their Idle wage income $\varepsilon$ more than their income from served trips. 

Given this setting, we can determine the equilibrium supply and demand in the system. Applying \Cref{eq:supply_equilibrium}, $L=\ell(e,J)$, to both period yields
\begin{align*}
    L_h= & \frac{(1-\tau)D_hp_h}{L_h} + (1+\varepsilon)J =  1-\tau+(1+\varepsilon)J, \\
    L_l= &\frac{(1-\tau)D_lp_l}{L_l} + (1 + \varepsilon)J=\frac{1-\tau}{4}+(1+\varepsilon)J,
\end{align*}
where we have replaced $e$ according to \Cref{eq:expectedEarnings}, we have used the definition of $D_h$ and $D_l$, and we have also replaced the optimal prices. These expressions imply that
\begin{align*}
    D_h=1-\tau+(1+\varepsilon)J \quad \text{and} \quad
    D_l=\frac{1-\tau}{8} + \frac{1+\varepsilon}{2}J.
\end{align*}
Hence, the platform's profits for each period, according to \Cref{eq:profit}, are given by 
\begin{align} 
    \profit_h&=\tau D_h - J D_h = (2+\varepsilon)\tau J + \tau - \tau^2 -J -(1+\varepsilon)J^2 \label{Eq:Utilities,h,l},\\ 
    \profit_l&= \tau D_l - J D_l = \frac{2+\varepsilon}{4} \tau J + \frac{\tau - \tau^2}{16} - (1+\varepsilon)J^2 - \frac{J}{4}  \label{Eq:Utilities,h,l2}.
\end{align}
Let us consider for a moment the hypothetical case in which the platform can set a different Idle wage in each period. In this case, we know that the optimal solution 
is to set $\tau=1$, and it is straightforward to optimise with respect to $J$ to obtain $J_h=1/2$ and $J_l=1/8$. Aligned with intuition, the optimal Idle wage is higher during the high-demand period when drivers are needed the most.

We now come back to the case in which the platform must decide a single $J$ (and a single $\tau$) for both periods. The total profit, given $J$ and $\tau$, is computed by adding up the two formulae in Eqs. \eqref{Eq:Utilities,h,l} and \eqref{Eq:Utilities,h,l2},
\begin{equation*}
    \profit =\frac{17}{16} \tau - \frac{17}{16} \tau^2 -2(1+\varepsilon)J^2 - \frac{5}{4}J + \frac{10+5\varepsilon}{4}J\tau.
\end{equation*}
We can consider three sub cases:
\begin{enumerate}
    \item [(a)] When the platform pays drivers only by using the Idle wage, i.e., $\tau=1$, the optimal Idle wage is $J=\frac{5}{16}$. In this case, the platform's profit is $(1+\varepsilon)\frac{25}{128}$.
    \item [(b)] When the platform cannot use the Idle wage, i.e., $J=0$, the optimal solution is to use $\tau=\frac{1}{2}$. In this case, the platform's profit is $\frac{34}{128}$.
    \item [(c)] When the platform can use both sources, i.e., $J$ might be greater than zero and $\tau$ lower than 1, the optimal solution is given by
\begin{equation}
  J= \begin{cases}
      
		\frac{85 \varepsilon}{4 (36 + 36 \varepsilon - 25 \varepsilon^2)}  & \mbox{if } \varepsilon \leq 0.72, \\
		0.3125 & \mbox{if } \varepsilon > 0.72,
  \end{cases}
\end{equation},

\begin{equation}\label{Eq:OptimalTauCounterExample}
    \tau= \begin{cases}
		\frac{18 + 43 \varepsilon}{36 + 36 \varepsilon - 25 \varepsilon^2}  & \mbox{if } \varepsilon \leq 0.72, \\
		1 & \mbox{if } \varepsilon > 0.72,
\end{cases}
\end{equation}

and the platform's profit is
\begin{equation}
\profit=
\begin{cases}
    		\frac{153 (1 + \varepsilon)}{16 (36 + 36 \varepsilon - 25 \varepsilon^2)} & \mbox{if } \varepsilon \leq 0.72, \\
		\frac{25}{128}(1+ \varepsilon) & \mbox{if } \varepsilon > 0.72.
  \end{cases}
\end{equation}

\end{enumerate}

Note that all three cases correspond to optimising a quadratic problem. The first two are trivial since they become unidimensional. Subcase (c) yields a system of two linear equations. The threshold $\varepsilon=0.72$ appears because for greater values, the first row in Eq. \eqref{Eq:OptimalTauCounterExample} becomes greater than $1$ (which is unfeasible for $\tau$).

We first observe that if the platform must choose between using the Idle wage only (sub case (a)) or not using the Idle wage altogether (sub case (b)) then the former dominates the latter when drivers are sufficiently risk-averse ($\varepsilon\geq 0.36$). Importantly, when drivers' risk-aversion is not strong ($\varepsilon<0.36$), the platform may prefer not to use the Idle wage and rely solely on the commission from trips served. Second, consider now sub case (c) which characterises the optimal global solution. The analysis shows that the platform prefers to only use the Idle wage (i.e., to set $\tau =1$) when drivers' risk-aversion is sufficiently high ($\varepsilon>0.72$); else, the platform sets $J>0$ and $\tau<1$. Note that $\varepsilon$ is quantifying the trade-off mentioned above: the greater its value, the more risk-averse the drivers, and therefore, the most preferable for the platform to pay via $J$; analogously, if $\varepsilon$ diminishes, then leveraging the flexiblity of paying through $e$ becomes more important.

We note that the larger the changes in the demand, the more important the role played by $\tau$. For instance, if we repeat the example but with $D_l=L_l/4$ instead of $L_l/2$ (for the same price $p=1/2$), we would need $\varepsilon>1.22$ to obtain $\tau=1$.

This example illustrates why having different demand across periods makes $\tau$ a useful and critical tool for the platform: without it ($\tau=1$), drivers' earnings are not affected by the price nor the demand (because $e$ becomes $0$ in \Cref{eq:expectedEarnings}) and, therefore, the labour supply would be constant regardless of the demand level. Dynamic pricing is known to be a powerful tool to adjust the supply to the demand levels, and it becomes viable precisely when $\tau$ is lower than 1. On the other hand, the example also shows that even when risk-aversion is not strong (the context of low but positive $\varepsilon$), the platform can still leverage the Idle wage to optimise its profit (by setting $J>0$ and $\tau<1$ as in sub case (c)).

\subsection{A Multi-Period Model with Idle Wage Constraints }\label{sec:multiperiodintervals}
In this section, we propose an intermediate version regarding the flexibility of 
$J$. A fully dynamic $J$, potentially changing by the hour, could undermine one of the primary objectives of the Idle wage policy: reducing the uncertainty faced by drivers. Conversely, an Idle wage that remains constant across different time periods may be ineffective in managing demand-supply mismatches.
 In the alternative proposed in this section, $J$ can be defined per hour, but we impose constraints on the extent to which $J$ can vary. 

One possible constraint is to allow the Idle wage to be constant in a block of hours, ensuring that drivers earn a certain income during that block. For example, the platform offers 8 hours in a day during which drivers have access to the same Idle wage each hour. Another possibility is to impose a minimum wage. For example, the total earnings of a driver via the Idle wage are above a certain preset lower bound, e.g., 8 hours times the minimum hourly wage. The specific choice of the constraint imposed on the Idle wage would depend on the given context and what policymakers deem more appropriate. In this section, we propose a general methodology that enables the assessment of different policies aimed at imposing restrictions on how a platform can set the Idle wage.

In what follows, we present a problem formulation for the case in which the platform ensures a minimum wage via the Idle wage in certain blocks of hours. We note that this formulation can readily be adjusted to assess other possible policies. We assume each time period represents an hour, and that each driver works for  $b$ hours in a day. We will impose that the platform must offer at least one block of $b$ hours in a day during which drivers' total Idle wage is above certain fixed parameter $\minwage\geq0$. Because policymakers may be concerned with drivers not having economically viable options that allow them to have enough time to rest, we will also allow for the block $b$ to be split into two sub-blocks of $b_1$  and $b_2$ hours each with a gap of $d$ hours between the end of block 1 and the beginning of block 2. For example, $b$ could be 8 hours, corresponding to a standard workday, and $b_1$ and $b_2$ could both be 4 hours with a gap of 8 hours. The latter prevents undesirable situations in which a driver is, for example, incentivised to work for 1 hour every 3 hours.

 Let us utilise the convention that  $[N] = \{1,\ldots,N\}$ for any integer $N$ and that  
 the day is cyclic, that is, with the identification $24+1 = 1$, $24+2 = 2$ and $24+3 = 3$. We can identify the admissible blocks with the initial hours as follows:
\begin{equation*}\label{eq:AdmissibleBlocks}
   \mathcal{B} = \left\{ (h_1,h_2) \in [24]\ :\ h_1+b_1-1<h_2\text{ and } (h_2+b_2-1)<h_1\right\}.
\end{equation*}
Then, with the convention that $J_{25} = J_1$, $J_{26} = J_2$ and $J_{27} = J_3$, the final constraint is given by
\begin{equation}\label{eq:MinimumWageConstraintc}
   \exists (h_1,h_2)\in \mathcal{B},\quad \sum_{h=h_1}^{h_1+b_1-1} J_h +  \sum_{h=h_2}^{h_2+b_1} J_h\geq \minwage.
\end{equation}
Defining, 
\begin{equation*}\label{eq:MinimumWageConstraint}
  M(\mathbf{J}) \triangleq \max_{(h_1,h_2)\in \mathcal{B}}\left\{ \sum_{h=h_1}^{h_1+3} J_h +  \sum_{h=h_2}^{h_2+3} J_h\right\},
\end{equation*}
we can write the platform's problem as
\begin{equation}\tag{$\mathcal{P}_{\min}$}\label{eq:Opt-MultiPeriod-MinimumWage}
    \begin{array}{cl}
        \displaystyle\max_{p,J,\tau,e,I,L,Q} & \sum_{h=1}^{24}\theta(p_h,J_h,\tau, e_h,I_h,L_h,Q_h) \\
         s.t. &\left\{\begin{array}{l}
         p,J\geq 0, \tau \in [0,1],\\
           M(\mathbf{J})\geq \minwage,\\
        \forall h\in [24],\,\,(e_h,I_h,L_h,Q_h)\in\Eqset(p_h,J_h,\tau).
        \end{array}\right.
    \end{array}
\end{equation}
The following result establishes that problem \eqref{eq:Opt-MultiPeriod-MinimumWage} is well posed.

\begin{theorem}\label{thm:Existence-minimumWage}
    Suppose that there exists one feasible solution of the multi-period problem without minimum wage constraints and with $J=0$, such that
    \[
    \exists (h_1,h_2)\in \mathcal{B},\quad \sum_{h=h_1}^{h_1+3} e_h +  \sum_{h=h_2}^{h_2+3} e_h\geq \minwage,
    \]
    with the convention $e_{25} = e_1$, $e_{26} = e_2$ and $e_{27} = e_3$. Then, Problem \eqref{eq:Opt-MultiPeriod-MinimumWage} admits a solution.\\
    
    Moreover, if there is no aversion to risk, then the solution $(p^{\ast},J^{\ast},\tau^{\ast},e^{\ast},I^{\ast},L^{\ast},Q^{\ast})$ of \eqref{eq:Opt-MultiPeriod-MinimumWage} has the same objective value as the reference point, that is,
    \[
    \sum_{h=1}^{24}\theta(p_h,J_h,\tau, e_h,I_h,L_h,Q_h) = \sum_{h=1}^{24}\theta(p_h^{\ast},J_h^{\ast},\tau^{\ast}, e_h^{\ast}, I_h^{\ast},L_h^{\ast},Q_h^{\ast}).
    \]
\end{theorem}
\begin{proof}
The existence of a solution is a direct consequence of Theorem \ref{thm:ExistenceEquil}. We know that, period by period, we can modify $J_h$ and find a new equilibrium, so we can increase $J_h$ till $M(\bold{J})\geq J_{\min}$. 

If drivers are risk-averse, one can reduce $e$ by increasing $\tau$, and compensate exactly via $J$, so that the same payment is kept in every period.  This will produce the same set of equilibria, keeping the value function unaltered.
\end{proof}

The proof of Theorem \ref{thm:Existence-minimumWage} is based on the fact that, when transforming the expected earnings $e$ in Idle wage $J$, the obtained configuration remains feasible. When drivers are risk averse, the swapping $e$ to $J$ will attract more drivers, and thus for every period $i=h\in [24]$, one has that $I_h^{\prime} \geq I_h$ even if $p_i^{\prime} = p_i$, and so we get
    \[
    Q_h^{\prime} = D(p_h,T(I^{\prime}_h)) \geq D(p_h,T(I_h)) = Q_h
    \]
This means that the quality of service improves, the surplus of riders is higher (for the same price, they get better service) and more trips are served. However, the labour force also increases, so neither the profit of the company nor the welfare are necessarily better. In the first case, the company needs to cover a higher cost of larger labour force (through $J_h$). In the second case, the system needs to absorb the social cost of more drivers.

On the other hand, we remark that the optimal solution of problem \eqref{eq:Opt-MultiPeriod-MinimumWage} also has $\tau=1$ in every period. Periods that are not part of the two 4-hours blocks are optimised independently, so Theorem \ref{thm:tau1isoptimal} applies to each of them. For periods $h$ that are part of those blocks, the same argument as previous cases hold: if $\tau<1$, one can replicate the same equilibrium with $\tau=1$ and an increased $J_h$, which would increase profit and not decrease welfare.

Finally, we note that the main innovation in problem \eqref{eq:Opt-MultiPeriod-MinimumWage} is the constraint given in \Cref{eq:MinimumWageConstraintc}. While this constraint captures a minimum wage condition, we could easily change to capture other potential policies. For example, we could add another combinatorial constraint that requires the existence of an 8 hrs block with two 4 hours sub-blocks that ensure a constant per hour Idle wage. 



\section{Numerical Experiments}\label{sec:NumericalExperiments}
In this section, we quantify the impact of $J$ by running numerical experiments. As the problem is non-convex, in order to ensure that the global optimum is reached we perform an exhaustive grid search, with steps 0.05 for $\tau \in [0,1]$, 0.05 for $J \in [0,2.8]$ (where the maximum value was obtained by inspection), and 0.01 for the other variables. For each potential combination of $p,T,J$, when we find more than one equilibrium, we keep the one maximising the objective function.

\subsection{Specific formulas and parameters}
In sections \ref{sec:model_single_period}-\ref{sec:multiperiod}, the maximisation problems are written using generic functions for $D_h(p,T), T(I), \ell(e,J)$. Following \textcite{yan2020dynamic}, we use the following formulas for each of them:
\begin{align}
    D_h(p,T)=&\lambda_h\frac{e^{u_h(p,T)}}{1+e^{u_h(p,T)}}, &\text{ with } & u_h(p,T)=\kappa + \beta_p \cdot p + \beta_T \cdot T \label{Eq:FormulaD} \\ 
    I(T)=&\left(\frac{T}{K_T}\right)^\frac{1}{\alpha_T} \label{Eq:FormulaT}\\
    \ell(e,J) = &A_h\cdot\left(\frac{\beta e + J}{1+\frac{1}{\varepsilon_L}} \right)^{\varepsilon_L} \label{Eq:Formulaell}
\end{align}

Let us describe these equations one by one. In Eq. \eqref{Eq:FormulaD}, $u_h(p,T)$ represents the utility obtained by the users. Therefore, $\lambda_h$ represents the maximum number of potential users in period $h$, and the greater the utility, the greater the demand. The constants $\beta_p, \beta_T$ are negative numbers, and $\kappa$ captures the base utility\footnote{In the original paper by \textcite{yan2020dynamic}, they actually use $\kappa+\kappa_h$, so that the base utility is period-dependent. However, they only report the average of those $\kappa_h$, which is why we have merged the sum into a single figure. Let us remark that different periods still exhibit completely different results thanks to the $\lambda_h$.}. The numeric value for all these constants is reported in the Appendix and come from \textcite{yan2020dynamic}.

Eq. \eqref{Eq:FormulaT} shows that $T$ indeed decreases with $I$. The constant $\alpha_T>1$ captures the fact that the marginal effect of every extra driver decreases. This is consistent with previous findings in the spatial aspects of on-demand mobility (\cite{besbes2021surge,fielbaum2023economies}): once the spatial area is reasonably covered by idle drivers, adding extra ones will not have such a relevant effect on diminishing waiting times.

Eq. \eqref{Eq:Formulaell} is where risk aversion comes into play. The parameter $\beta \leq 1$ represents risk aversion. If $\beta=1$, the numerator is exactly the expected earnings of the drivers. If $\beta<1$, it means that drivers penalise every expected dollar earned via $e$ because of its uncertainty. Therefore, in our simulations, we will compare the results for varying values of $\beta$. The parameter $A_h$ controls the maximum number of drivers per period, and $\varepsilon_L$ is the elasticity. The values for these parameters are also reported in the appendix. Only $A_h$ is varied in some experiments, to analyse numerically the result from Theorem \ref{Thm:JisZeroForLargeA}; when we vary $A_h$ below, we explicitly mention it.

Let us remark that all the assumptions explained in Section $\ref{sec:model_single_period}$ are directly verified by the functions described in Eqs. \eqref{Eq:FormulaD}-\eqref{Eq:Formulaell}.

\subsection{Results}
Let us begin analysing the effect of $J$ in the single-period context, considering the period $h=19$ in isolation. This is analysed in Figure \ref{fig:SinglePeriodDifferentJ}, where we show the best possible welfare/profit achievable for any given $J$. That is, $p$ and $\tau$ are optimised. Recall that if $J$ can be chosen freely, then in this single-period case it is optimal to select $\tau=1$, so we mark every point in the curve where $\tau=1$ is optimal. The most important conclusions from Fig. \ref{fig:SinglePeriodDifferentJ} are:

\begin{figure}[h]
    \centering
    \includegraphics[width=150mm]{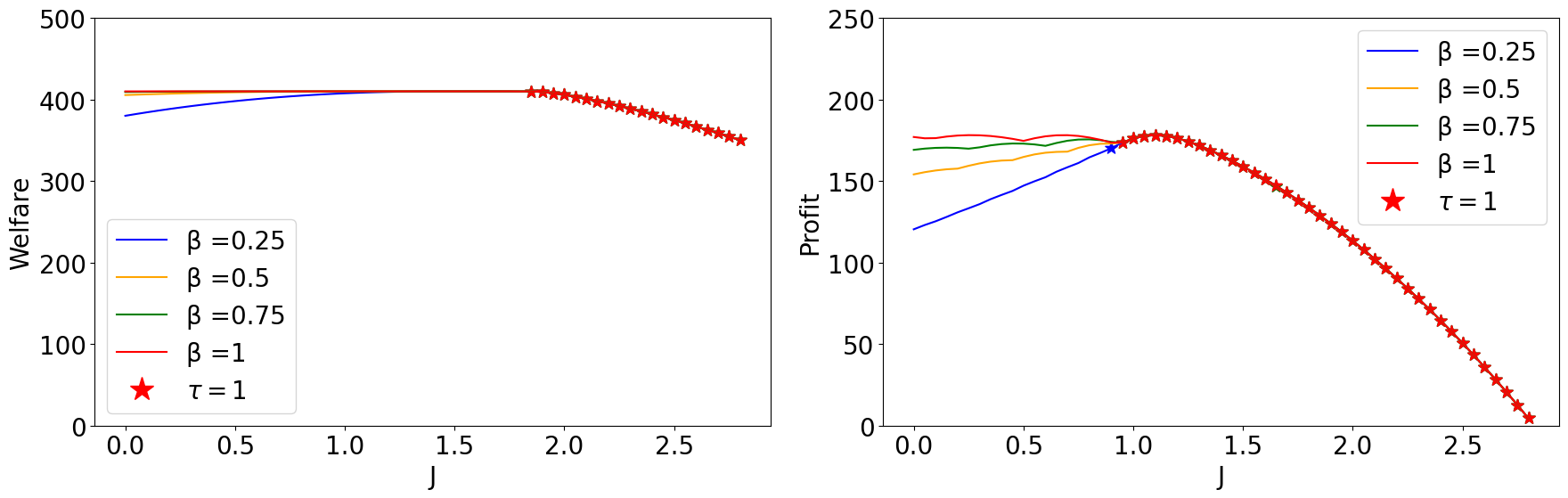}
    \caption{Optimal results for any given $J$ in the single-period scenario, when optimising welfare (left) or profit (right). We mark with a star the points where it is optimal to use $\tau=1$.}
    \label{fig:SinglePeriodDifferentJ}
\end{figure}

\begin{itemize}
    \item When $\beta<1$, the objective functions have an inverted U-Shape as a function of $J$. That is, it is better to have a positive Idle wage, as this is better for the drivers than earning the same amount via $e$. However, if $J$ is too large, too many drivers will be attracted rendering the solution suboptimal.
    \item As intuitive, $\tau=1$ is optimal when $J$ is large, because this is the situation where drivers are already receiving enough earnings through the Idle wage.
    \item In both graphs, the maximum is achieved within the ($\tau=1$)-range, as predicted by Theorem \ref{thm:tau1isoptimal}.
\end{itemize}

We know that if $J$ and $\tau$ have to take a unique value in a multi-period framework, then the optimal combination might include payments via trips (i.e. $\tau<1$) and through Idle wage (i.e. $J>0$). Moreover, from Theorem \ref{Thm:JisZeroForLargeA} we know that when there is a large number of available drivers (represented in this case by $A_h\gg 0$ in Eq. \ref{Eq:Formulaell}), $J$ becomes lower in the for-profit case. 

These two effects are combined in Table \ref{tab:AChangesTwoPeriods}. Therein, we simulate a two-periods model, considering the least and most demanded periods according to $\lambda_h$, which are 4 AM and 7 PM (i.e., periods $h=4$ and $h=19$), respectively. We slightly modify them together, to increase or decrease the number of available drivers, by changing the corresponding $A_h$. Table \ref{tab:AChangesTwoPeriods} contains the results for 6 different values of $\beta$, and where we jointly increase the values of $A_4$ and  $A_{19}$. Let us denote by $A_B$ the bi-dimensional vector $A_B=(A_4,A_{19})$.

We begin analysing the for-profit case in Table \ref{tab:AChangesTwoPeriods}. If $\beta$ is not close to 1, then the efficiency of paying through $J$ dominates the analysis, and we obtain that $\tau=1$. Only when drivers are not so risk-avert, i.e. when $\beta=0.95$, we find that $\tau$ might be lower than 1, and this happens exactly as $A_B$ increases. In other words, and as suggested by Theorem \ref{Thm:JisZeroForLargeA}, the greater the $A_B$, the more convenient to pay via $e$, because otherwise too many drivers will be connected and receiving Idle wage.

\begin{table}[H]
\centering
\begin{tabular}{|c|c|c|c|c|c|c|c|}
\hline
$\beta$ & $A_B$         & \multicolumn{3}{c}{Welfare} & \multicolumn{3}{|c|}{Profit} \\
     &   $(A_4,A_{19})$           & Optimal J & $\tau$ & Value    & Optimal J & $\tau$ & Value    \\
\hline
0.2  & (3.5, 44.0) & 1.5       & 0   & 420.15 & 1.1       & 1   & 181.6  \\
     & (4.0, 44.5) & 1.4       & 0   & 420.48 & 1.1       & 1   & 181.6 \\
     & (4.5, 45.0) & 1.4       & 0   & 420.8 & 1.1       & 1   & 181.7 \\
     & (5.0, 45.5) & 1.4       & 0   & 420.8 & 1.1       & 1   & 181.6  \\
     & (5.5, 46.0) & 1.4       & 0   & 421.3 & 1.1       & 1   & 181.5 \\
\hline
0.35 & (3.5, 44.0) & 1.3       & 0.1 & 420.6 & \multicolumn{3}{|c|}{Same as the scenario $\beta = 0.2$} \\
     & (4.0, 44.5) & 1.2       & 0   & 421.1 & \multicolumn{3}{|c|}{}                          \\
     & (4.5, 45.0) & 1.1       & 0   & 421.5 & \multicolumn{3}{|c|}{}                          \\
     & (5.0, 45.5) & 1         & 0   & 421.8 & \multicolumn{3}{|c|}{}                          \\
     & (5.5, 46.0) & 0.9       & 0   & 421.9 & \multicolumn{3}{|c|}{}                          \\
\hline
0.5  & (3.5, 44.0) & 1.2       & 0.1 & 420.6  & \multicolumn{3}{|c|}{Same as the scenario $\beta = 0.2$}   \\
     & (4.0, 44.5) & 1.1       & 0   & 421.2 &  \multicolumn{3}{|c|}{}  \\
     & (4.5, 45.0) & 1         & 0   & 421.8  &  \multicolumn{3}{|c|}{}  \\
     & (5.0, 45.5) & 1         & 0   & 422.3 &   \multicolumn{3}{|c|}{}  \\
     & (5.5, 46.0) & 0.9       & 0   & 422.7 &   \multicolumn{3}{|c|}{}  \\
\hline
0.65 & (3.5, 44.0) & 0.9       & 0.1 & 420.7 & \multicolumn{3}{|c|}{Same as the scenario $\beta = 0.2$}  \\
     & (4.0, 44.5) & 0.8       & 0   & 421.3  &   \multicolumn{3}{|c|}{}   \\
     & (4.5, 45.0) & 0.9       & 0.1 & 421.8 &   \multicolumn{3}{|c|}{}    \\
     & (5.0, 45.5) & 0.9       & 0.1 & 422.3 &   \multicolumn{3}{|c|}{}    \\
     & (5.5, 46.0) & 0.8       & 0.1 & 422.8 &   \multicolumn{3}{|c|}{}    \\
\hline 
0.8  & (3.5, 44.0) & 0.7       & 0.1 & 420.7 & \multicolumn{3}{|c|}{Same as the scenario $\beta = 0.2$}  \\
     & (4.0, 44.5) & 0.5       & 0   & 421.2 &   \multicolumn{3}{|c|}{}     \\
     & (4.5, 45.0) & 0.7       & 0.1 & 421.8 &  \multicolumn{3}{|c|}{}   \\
     & (5.0, 45.5) & 0.8       & 0.2 & 422.3 &    \multicolumn{3}{|c|}{}     \\
     & (5.5, 46.0) & 0.8       & 0.2 & 422.8 &    \multicolumn{3}{|c|}{}    \\
\hline    
0.95 & (3.5, 44.0) & 0.3       & 0   & 420.7 & 1.1       & 1   & 181.6  \\
     & (4.0, 44.5) & 0.6       & 0.1 & 421.2 & 1.1       & 1   & 181.6 \\
     & (4.5, 45.0) & 0.5       & 0.1 & 421.8 & 0.3       & 0.9 & 182 \\
     & (5.0, 45.5) & 0.4       & 0.1 & 422.3 & 0.3       & 0.9 & 182.2 \\
     & (5.5, 46.0) & 0.3       & 0.1 & 422.8  & 0.3 & 0.9 & 182.5\\
     \hline
\end{tabular}
\caption{Optimal value of $\tau$ and $J$, if they have to remain fixed throughout the day, for different combinations of $A_4, A_{19}$, and $\beta$.}   
\label{tab:AChangesTwoPeriods}
\end{table}

In order to analyse the welfare-optimisation scenario, it is convenient to first look into Figure \ref{Fig:ValueVSTauMultiPeriod}, where we depict the value of the two objective functions as a function of $\tau$, if a single $\tau$ must be used during the whole day. There we observe that the total welfare changes little with $\tau$. In Table \ref{tab:AChangesTwoPeriods}, we can observe that: if $\beta$ is low, we still obtain that greater values of $A_B$ lead to a lower $J$, for the exact same reasons discussed in the for-profit scenario. However, when $\beta$ is close to 1, there is no clear relationship between the two.

\begin{figure}[H]
\centering
\includegraphics[width=1\linewidth]{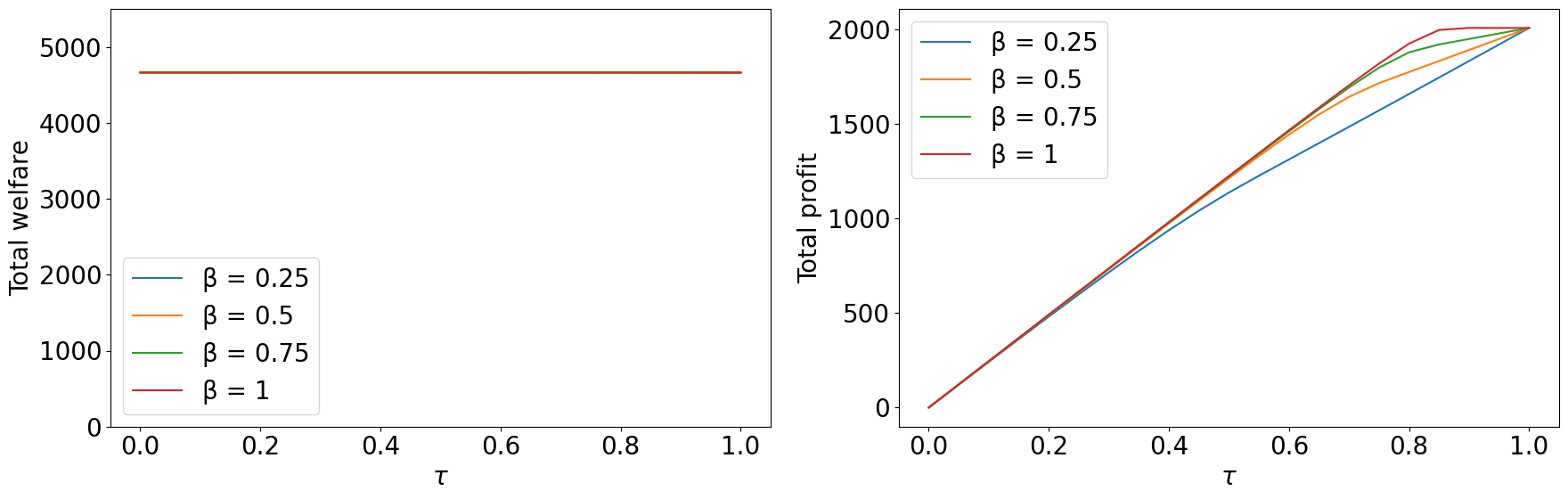}
\caption{Optimal results for any given $\tau$ in the full-day scenario, where each $J_h$ is optimised independently and $\tau$ is fixed.}
\label{Fig:ValueVSTauMultiPeriod}
\end{figure}

In the full-day scenario, we considered three potential alternatives about $J$: Fully flexible, fixed, or the double-block minimum wage constraint. Let us begin analysing the case in which $J$ can be adjusted by period. In this case, each period is optimised indepedently, so it is optimal to use $\tau=1$. This is shown in Fig. \ref{Fig:ValueVSTauMultiPeriod}. In the optimal welfare case, the objective function is almost constant, which is consistent with Theorem \ref{thm:tau1isoptimal} that predicts that the objective functions increase, but not necessarily strictly, with $\tau$. In the for-profit case, it does increase significantly with $\tau$; the lower the $\beta$ (i.e., the more risk-avert the drivers), the more relevant it is to utilize $\tau=1$. 

The need for adjusting to a fluctuating demand gets clearly illustrated in Fig. \ref{Fig:JPerHour}, where we show the optimal $J$ per hour. Note that high-demand periods require a greater $J$, i.e., to attract more drivers. It is worth noting that the for-profit scenario offers a substantially lower $J$ than when optimising welfare. That is, a monopolistic situation in which the platform may decide on the price would lead to a undersupplied system in which passengers wait more than what is socially optimal. We remark that, as in every period it is optimal to use $\tau=1$, then $\beta$ plays no role in this figure, which is why we show a single curve.

\begin{figure}[H]
\centering
\includegraphics[width=1\linewidth]{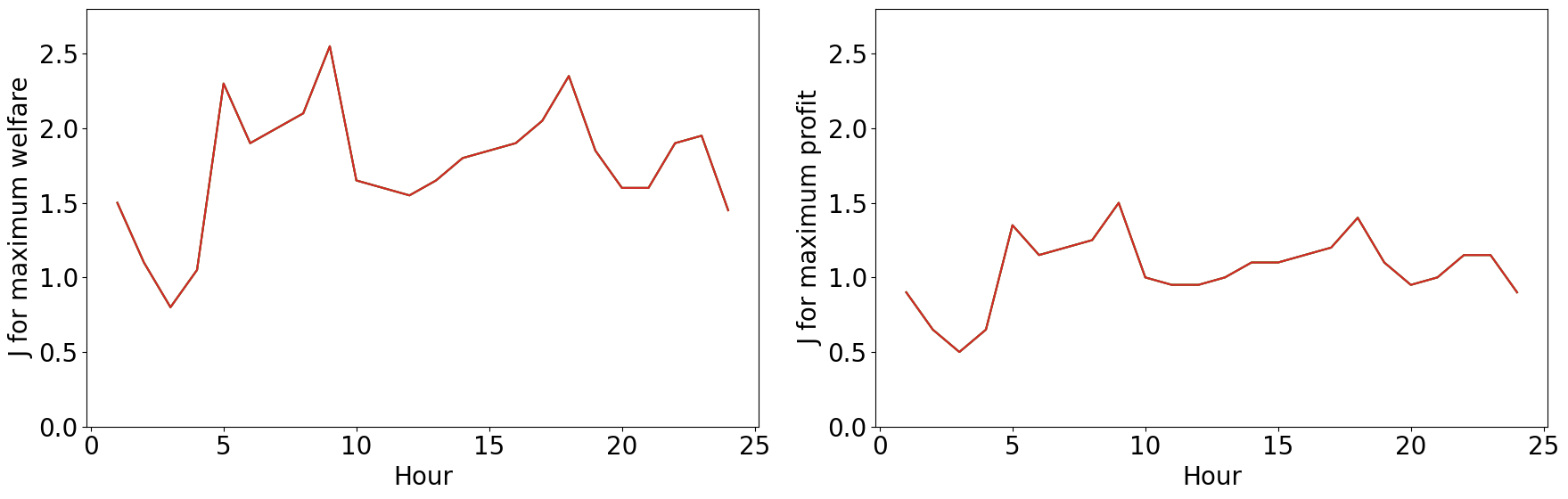}
\caption{Optimal $J_h$ per hour in the full-day scenario.}
\label{Fig:JPerHour}
\end{figure}

Fig. \ref{fig:fixedJandtauforwholedayscenario} shows the results when a single $J$ and $\tau$ must be decided for the whole day. As discussed in the two-periods scenario, here it is no longer optimal to use $\tau=1$. Moreover:
\begin{itemize}
    \item If $\beta=1$, i.e., drivers are risk-neutral, it is actually optimal to use $J=0$. This happens because $J$ presents no added benefit when $\beta=1$, and $e$ can capture the changes in price.
    \item If $\beta<1$, we keep some of the conclusions from the single-period scenario, namely: i) The objective functions have an inverted U-shape, and ii) If $J$ is large enough, it is convenient to use $\tau=1$.
\end{itemize}

\begin{figure}[h]
    \centering
    \includegraphics[width=150mm]{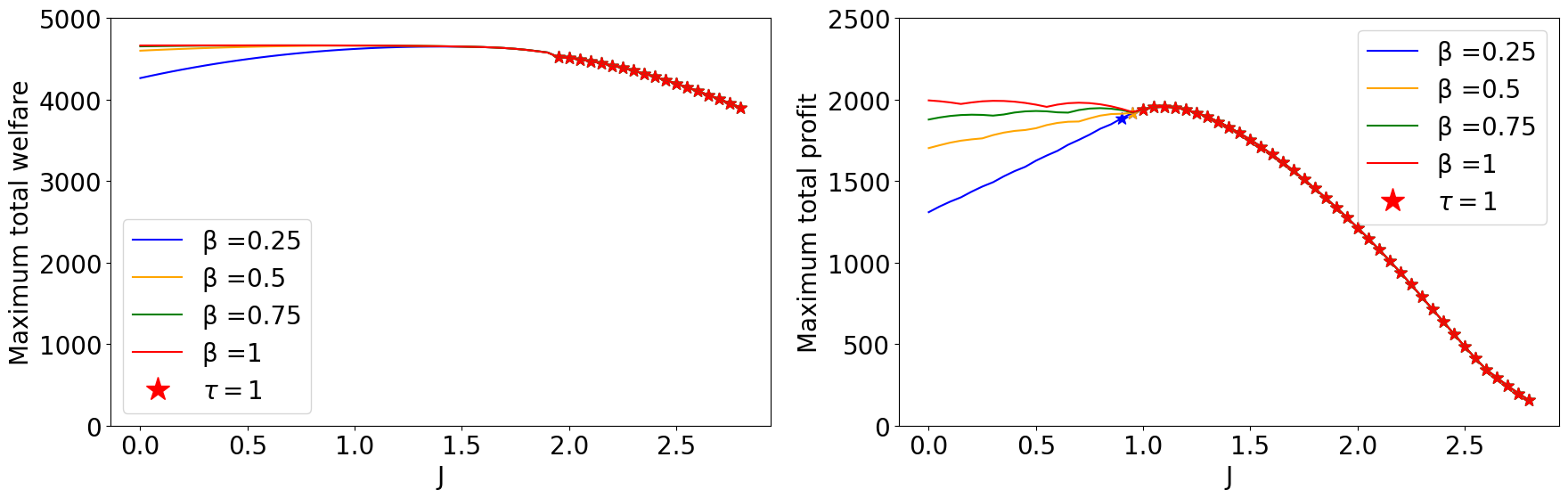}
    \caption{Optimal results when $J$ and $\tau$ take a single value in the full-day scenario. We mark with a start the values of $J$ for which it is optimal to use $\tau=1$.}
    \label{fig:fixedJandtauforwholedayscenario}
\end{figure}

Finally, we have the case in which the double-block minimum wage constraints must be satisfied, according to the model of Section \ref{sec:multiperiodintervals}. Let us remark that numerically solving problem \eqref{eq:Opt-MultiPeriod-MinimumWage} is very complex. Not only the equilibrium conditions are not linear, but the restrictions are combinatorial. In other words, there is no simple way to determine which two 4-hours blocks should be used. To face this, we design the following heuristic: We first solve each period independently and obtain a first vector $\bold{J_0}$. Within $\bold{J_0}$, we search for the two disjoint 4-hours blocks where $M(\bold{J_0})$ achieves its maximum. We fix those blocks and optimise accordingly. That is, denoting by $h_1(\bold{J_0}), h_2(\bold{J_0})$ the start of the interval, we find the optimal solution to:

\begin{equation}\label{eq:Opt-MultiPeriod-MinimumWageHeuristi}
    \begin{array}{cl}
        \displaystyle\max_{p,J,\tau,e,I,L,Q} & \sum_{h=1}^{24}\theta(p_h,J_h,\tau, e_h,I_h,L_h,Q_h) \\
         s.t. &\left\{\begin{array}{l}
         p,J\geq 0, \tau \in [0,1],\\
         \sum_{h=h_1(J_0)}^{h_1(J_0)+3}J_h+\sum_{h=h_2(J_0)}^{h_2(J_0)+3}J_h\geq J_{\min},\\
        \forall i\in [24],\,\,(e_i,I_i,L_i,Q_i)\in\Eqset(p_i,J_i,\tau).
        \end{array}\right.
    \end{array}
\end{equation}

This scenario is illustrated in Fig. \ref{Fig:MinWage}, where we use $\beta=0.25$. Therein, we compare the unconstrained scenario (where $J$ can be optimised per period) against the constrained one with different values for the minimum wage. As discussed in Section \ref{sec:multiperiodintervals}, this constraint might reduce the value of the objective function, which happens in every case in the for-profit scenario, and when the minimum wage is too large in the welfare scenario.

\begin{figure}[H]
\centering
\includegraphics[width=1\linewidth]{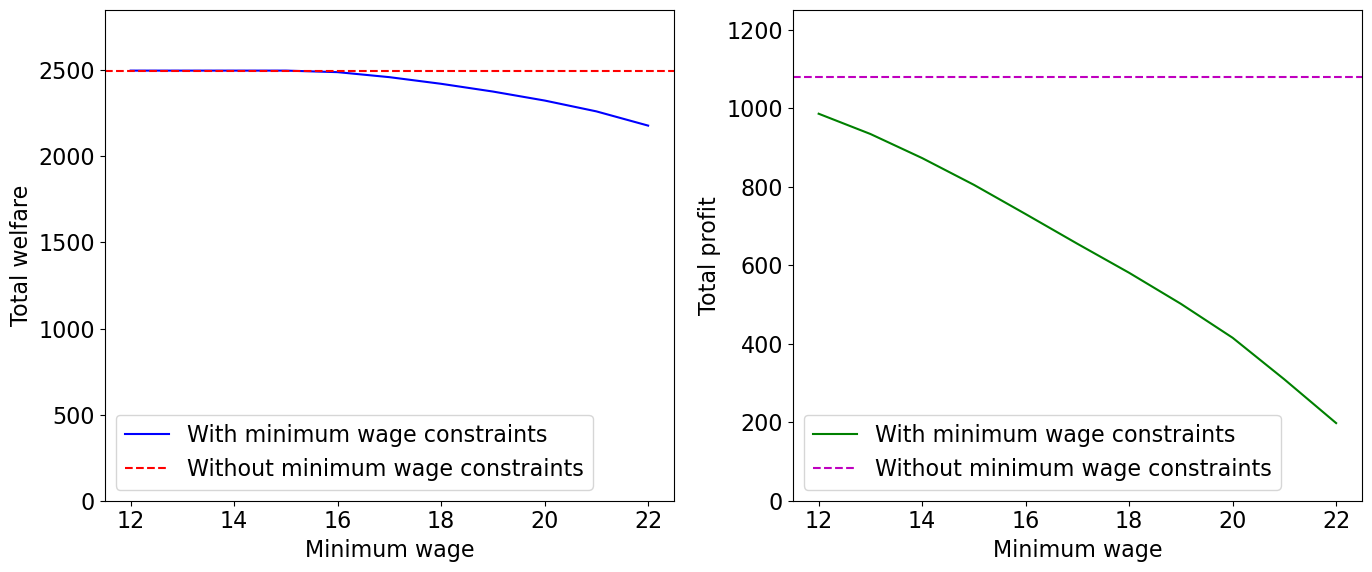}
\caption{Optimal welfare (left) and profit (right), with or without the minimum wage constraints.}
\label{Fig:MinWage}
\end{figure}

\section{Conclusions}\label{sec:conclusions}
In recent years, many countries have tried different policies to regulate the relationship between ride-hailing drivers and platforms. However, there is still no clear answer in terms of which measure could provide better job quality to the drivers, and still make the business profitable. In this paper, we analyse one potential policy, namely including and Idle wage for the drivers, so that when they are connected but not transporting passengers they can still receive some income. Ours is a macroscopic economic analysis, focusing on how are the supply-demand equilibria affected by the introduction of the Idle wage, when a crucial assumption of risk-aversion is considered.

We prove that in a single-period scenario, paying only through Idle wage is always optimal, both in for-profit and welfare-maximisation contexts. This result occurs because drivers prefer to earn money through the Idle wage, as this is less random, but for the platform it is indifferent how to pay. However, the Idle wage cannot deal with the demand fluctuations, as it is not affected by dynamic pricing, which is the main tool that platforms have to match the supply and demand levels. Thus, in a multi-period scenario, it is crucial to decide on a rule about how much the Idle wage could change. We study two extreme alternatives, namely the Idle wage being fully flexible or no flexible at all, and we propose an intermediate alternative where the Idle wage needs to build an exogenous minimum wage through reasonable intervals.

While ours is a simplified model, that would require several adjustments to be implemented, it does provide three strategic conclusions:
\begin{itemize}
    \item The Idle wage can be an effective way to make the drivers' earnings more predictable, while increasing the efficiency of the system.
    \item The trade-off between a stable income and the ability to match demand and supply is one of the main obstacles.
    \item A relevant concern can arise when the number of available drivers is large, as even a small Idle wage could reduce profit too dramatically.
\end{itemize}

From a future research perspective, the most important direction is about drivers' behaviour. Previous studies have shown that drivers do coordinate (e.g., \cite{schroder2020anomalous}), so this model should be combined with game theory tools in order to capture these more complex decisions. Moreover, here we have focused on ride-hailing companies; delivery drivers face a similar problem in terms of job quality, but the economic structure is different, so extending this method to these services is yet another relevant avenue for future studies. On a different note, here we have assumed that there is a single platform that can act in a monopolistic way: when the competition between platforms is included, they compete for users but also for the drivers, so the Idle wage can play a role there as well. 

Finally, drivers are known to be diverse, with some of them driving as their main job and others to complement their salaries (\cite{fielbaum2021sharing,ramezani2022empirical}); this suggests that some of them might be risk averse and others risk seekers. It could be possible to design a system where different options are offered to each driver. However, this is not a trivial task, as the power imbalance between drivers and platforms can lead to a situation where all drivers are led to opt for the same option, masking a situation where there is no actual choice. Contract theory might offer some tools to face such a problem (\cite{bolton2004contract}).

\section*{Acknowledgments}
This paper will be presented at the Conference in Emerging Technologies in Transportation Systems (TRC-30),
September 02-03, 2024, Crete, Greece. Dr David Salas has been partially supported by the Center of Mathematical Modeling CMM, grant FB210005 BASAL funds for centers of excellence (ANID-Chile) and by the project FONDECYT 11220586 (ANID-Chile).

\newpage
\printbibliography

\appendix
\section*{Appendix 0: Summary of Variables and Parameters}\label{app:summ_var}
\Cref{Tab:BasicNotation} summarizes the notation we use throughout the paper. 

\begin{table}[h]
    \centering
    \begin{tabular}{|c|c|c|}
    \hline
    \textbf{Symbol} & \textbf{Meaning} &\textbf{Range} \\
    \hline
    $p$ & Price per trip charged to the users & $[0,\infty)$       \\
    $\lambda$ & Riders connecting to the app per time unit & Constant\\
    $Q$ & Trips served per time unit & $[0,\lambda]$ \\
    $t$ & In-vehicle time to serve a user & Constant \\
    $T$ & Waiting time & $(0,\infty)$ \\
    $L$ & Supply size & $[0,\infty)$ \\
    $\tau$ & Percentage of the fare kept by the company & $[0,1]$ \\
    $I$ & Number of idle driver & $[0,\infty)$ \\
    $D$ & Demand function & $[0,\infty)$ \\
    $\ell$ & Supply function & $[0,\infty)$ \\
    $e$ & expected earnings from trips & $[0,\infty)$\\
    $J$ & Idle wage & $[0,\infty)$\\
    \hline
    \end{tabular}
    \caption{Basic notation of the model}
    \label{Tab:BasicNotation}
\end{table}

\section*{Appendix 1: Proof of \Cref{thm:ExistenceEquil}}

If $J = 0$, we can always choose $(e,I,L,Q) = (0,0,0,0)$ as an equilibrium. Thus, let us consider the case where $J>0$, and let us take $z\in (0,\infty)$ as the pickup time. We can compute 
\begin{align*}
Q_z &= D(p,z),\\
I_z &= T^{-1}(z),\\
L_{1}(z) &= I_z + (t+z)Q_z,\\
e_z &= (1-\tau)p\frac{Q_z}{L_{1}(z)},\\
L_{2}(z) &= \ell(e_z, J_z).
\end{align*}
Note that $Q_z,I_z,e_z,L_1(z),L_2(z)> 0$.
It is not hard to see that $z\in( 0,+\infty)\mapsto L_i(z)$ for $i=1,2$ are continuous mappings. Now, let us consider the following two limit cases:
\begin{itemize}
    \item $z\to \infty$: In this case, $Q_z\to 0$ and $I_z\to 0$. Then, applying \eqref{eq:QuickEvanescent-T}, we get that $L_1(z) \to 0$. In parallel, one has that
    \[
    0\leq e_z = (1-\tau)p\frac{Q_z}{L_1(z)} \leq (1-\tau)p\frac{1}{t+z} \to 0.
    \]
    Thus, continuity of $\ell$ yields that $L_2(z) = \ell(e_z,J) \to \ell(0,J)>0$. We get that there must exist $z^+\in (0,\infty)$ large enough such that $L_2(z^+) - L_1(z^+) > 0$.
    \item $z\to 0$: In this case, $Q_z \to \sup_T D(p,T) = D(p,0)<\infty$ and $I_z\to \infty$. Then, $L_1(z) \to \infty$. We get that 
    \[
    0\leq e_z \leq (1-\tau)p\frac{Q_z}{L_1(z)} \leq (1-\tau)p\frac{D(p,0)}{I_z} \to 0.
    \]
    Thus, continuity of $\ell$ yields that $L_2(z) = \ell(e_z,J) \to \ell(0,J)<\infty$. We get that there must exist $z^-\in (0,z^+)$ small enough such that $L_2(z^-) - L_1(z^-) < 0$.
\end{itemize}
Applying the well-known Bolzano's theorem (also known as the intermediate value theorem), continuity of $L_1(z) - L_2(z)$ entails that there exists $z_0\in (z^-,z^+)$ such that $L_2(z_0) - L_1(z_0) = 0$. Then, setting $L_{z_0}$ the common value of $L_1(z_0)$ and $L_2(z_0)$, we get that $(e_{z_0}, I_{z_0},L_{z_0},Q_{z_0})$ is an equilibrium in $\mathrm{Eq}(p,J,\tau)$. 

\section*{Appendix 2: Proof of \Cref{thm:OneVariableSuffices}}
Let us suppose that $Q_1 = Q_2$, and let $\bar{Q}$ be the common value.\\ 
    
    Noting that the function $D(p,\cdot)$ is one-to-one on $[0,\infty]$, we can define $z = D(p,\cdot)^{-1}(Q)$, with the convention that $\infty = D(p,\cdot)^{-1}(0)$. Then, the number of idle drivers is uniquely given by $\bar{I} = T^{-1}(z)$, with the convention that  $0 = T^{-1}(\infty)$. Then, the labour force $\bar{L}$ is uniquely determined by equation \eqref{eq:balanceLaborForce}. If $\bar{Q}= 0$, then $\bar{L} = 0 + (t+\infty)\cdot 0 = 0$. Finally, $\bar{e}$ is computed using \eqref{eq:expectedEarnings}, in terms of $(p,\bar{L},\bar{Q},\tau)$ (recalling the convention that $\bar{e} = 0$ if $\bar{L} = 0$). Thus, both tuples must coincide with $(\bar{e},\bar{I},\bar{L},\bar{Q})$.

\section*{Appendix 3: Proof of \Cref{Thm:ExistenceOfSolution}}
Let us first show that \eqref{eq:Problem-single-period} has a closed feasible set. Indeed, let us consider a sequence $\bold{x}_n = (p_n,J_n,\tau_n,e_n,I_n,L_n,Q_n)$ of feasible points converging to a limit tuple $\bold{x} = (p,J,\tau,e,I,L,Q)$. Clearly, $p,J\geq 0$ and $\tau \in [0,1]$, so we only need to prove that $(e,I,L,Q) \in \mathcal{E}(p,J,\tau)$. Again, since $e_n,I_n,L_n,Q_n\geq 0$ for all $n\in \mathbb{N}$, then it is direct that $e,I,L,Q\geq 0$. Moreover, by continuity of the supply function $\ell$, we get that $L = \ell(e,J)$. For the other three equations, we need to consider two different cases:
\begin{enumerate}
    \item \textbf{First case, $Q>0$:} In this case, we have that for $n\in\mathbb{N}$ large enough, $Q_n>0$. Moreover, 
    \[
    L = \lim_n L_n = \lim_n I_n + (t + T(I_n))Q_n \geq \lim_n tQ_n > 0.
    \]
    Thus, since $L>0$, we also have that  for $n\in\mathbb{N}$ large enough, $L_n>0$ and so, we can write
    \[
    e_n = (1-\tau_n)p_n \frac{Q_n}{L_n}.
    \] 
    Since the expression of the right-side is continuous at $(p,L,Q,\tau)$ with $L>0$, we deduce that
    \[
    e = (1-\tau)p \frac{Q}{L}.
    \]
    Finally, continuity also entails that $T(I_n)\to T(I)$. Note that $T(I)<\infty$. Indeed, if not, given that $Q>0$ we would have that
    \[
    L = \lim_n I_n (t + T(I_n))Q_n = I + (T + \infty) Q = \infty,
    \]
    which would be a contradiction. Therefore, continuity entails that $L = I + (t + T(I))Q$ and that $Q = D(p,T(I))$. With this, all equations in \eqref{eq:EquilibriumSet} are verified and we conclude that $(e,I,L,Q)\in \mathcal{E}(p,J,\tau)$.
    \item \textbf{Second case, $Q=0$:} Assume first that $(Q_n)$ admits a subsequence where $Q_{n_k} = 0$, for every $k\in \mathbb{N}$. Then, $(0,0,0,0) = (e_{n_k},I_{n_k},L_{n_k},Q_{n_k}) \to (e,I,L,Q)$. Moreover, $J_{n_k} = 0$ since for every $a\geq 0$, $\ell(a,\cdot)^{-1}(0) = 0$. Then, $J = 0$. Then, regardless the values of $p$ and $\tau$, we have that $(e,I,L,Q) = (0,0,0,0) \in \mathcal{E}(p,0,\tau)$.\\
    
    Assume now that $Q_n>0$ for all $n\in \mathbb{N}$. Again, a continuity argument yields that $e,I,Q,L\geq 0$ and that $L = \ell(e,J)$. Let $z_n = D(p_n,\cdot)^{-1}(Q_n)$. If $(z_n)$ admits a bounded subsequence, let us say with $\sup_k z_{n_k} = M < \infty$, then we could write
    \[
    Q = \lim_k Q_{n_k} = \lim_k D(p_{n_k},z_{n_k}) \geq \lim_k D(p_{n_k}, M) = D(p,M)>0.
    \]
    Thus, we get that $z_n \to \infty$. Then, we can write that $I_n = I(z_n) \to 0$ and therefore, that
    \[
    L = \lim_n I_n + tQ_n + z_nD(p_n,z_n)\leq \lim_n z_nD(0,z_n) = 0,
    \]
    where the inequality follows from the fact that $D(p_n,z_n)\leq D(0,z_n)$, for every $n\in\mathbb{N}$, and where the last limit follows from the evanescent property \eqref{eq:QuickEvanescent-T}. Recalling that $\infty\cdot 0 = 0$, we deduce that $Q = D(p,T(I))$ and that $L = I + (t+T(I))Q$.

    To finish, we only need to check that $e = 0$. But this is trivial since $L = 0$, and so,  $0 = \ell(e,J)$ yields that $e = 0$. This proves that $(e,I,Q,L)\in \mathcal{E}(p,J,\tau)$.
\end{enumerate}
 We have proven that \eqref{eq:Problem-single-period} has a closed feasible set, that we will denote as $\bold{X}$. Now, we claim that there exists $M\geq 0$ large enough such that
 \begin{equation}\label{eq:Big-M}
\sup\{ \theta(\bold{x})\ :\ \bold{x}\in \bold{X}, \|\bold{x}\|\leq M\} = \sup\{ \theta(\bold{x})\ :\ \bold{x}\in \bold{X}\}.
 \end{equation}
 Suppose the contrary. Then, there exists a sequence $(\bold{x}_n)\subset X$ with $\theta(\bold{x}_n)\to \sup\{ \theta(\bold{x})\ :\ \bold{x}\in \bold{X}\}$ and $\|\bold{x}_n\| \to \infty$. 

 We claim that the sequences $(J_n)$ and $(p_n)$ must be bounded.
 \begin{enumerate}
     \item \textbf{Case 1, $\theta(\bold{x}) = \Pi(\bold{x})$:} Suppose first that $J_n\to \infty$. Then, $L_n\to \infty$. Recalling that the function $p\mapsto pD(p,0)$ is bounded by \eqref{eq:QuickEvanescent-p}, we have that
 \begin{align*}
     \Pi(\bold{x}_n) = \tau_np_nQ_n - J_nL_n \leq p_n D(p_n,0) - J_nL_n \to -\infty,
 \end{align*}
 which is a contradiction with the fact that $\Pi(\bold{0}) = 0$. Thus, $(J_n)$ must be bounded. Now, suppose that $p_n\to \infty$. Then, again using \eqref{eq:QuickEvanescent-p}, we have that
 \[
  \Pi(\bold{x}_n) = \tau_np_nQ_n - J_nL_n \leq p_n D(p_n,0)\to 0.
 \]
 Thus, in this case, the optimal value is reached at $\bold{x} = \bold{0}$, and $M$ can be taken as $M=0$. This is again a contradiction, which yields that $(p_n)$ is also bounded.

 \item \textbf{Case 2, $\theta(\bold{x}) = W(\bold{x})$:} Suppose first that $J_n\to \infty$. Then, $L_n\to \infty$. Recalling that the function $p\mapsto pD(p,0)$ is bounded by \eqref{eq:QuickEvanescent-p} and that $D(\cdot,0)$ is integrable, we have that
 \begin{align*}
     W(\bold{x}_n) = S(p_n,T(I_n)) + p_nQ_n - C(L_n) \leq \int_0^{\infty}D(z,0)dz + p_nD(p_n,0) - C(L_n) \to -\infty.
 \end{align*}
 The limit follows from the fact that $C(L_n) \to \infty$. This is a contradiction with the fact that $W(\bold{0}) = 0$. Thus, $(J_n)$ must be bounded. Now, suppose that $p_n\to \infty$. Then, again we have that
 \begin{align*}
     W(\bold{x}_n) = S(p_n,T(I_n)) + p_nQ_n - C(L_n) \leq \int_{p_n}^{\infty}D(z,0)dz + p_nD(p_n,0)\to 0.
 \end{align*}
 Thus, in this case, the optimal value is reached at $\bold{x} = \bold{0}$, and $M$ can be taken as $M=0$. This is again a contradiction, which yields that $(p_n)$ is also bounded.
 \end{enumerate}
The claim is then proven. Let $\bar{p} = \sup_n p_n$ and $\bar{J} = \sup_n J_n$. We have that $e_n \leq \bar{p}$, and so, $L_n \leq \ell(\bar{p},\bar{J})$. This yields that $I_n\leq \ell(\bar{p},\bar{J})$ and $Q_n \leq D(\bar{p},0)$. Thus, $(\bold{x}_n)$ is bounded, which is a contradiction. We conclude that \eqref{eq:Big-M} holds.

Set $M$ as given by \eqref{eq:Big-M}, and let $\bold{Y} = \{\bold{x}\in \bold{X},\|\bold{x}\|\leq M \}$ which is now  a compact set. Since $\theta(\cdot)$ is a continuous function, there must exists $\bold{x}^{\ast}\in \bold{Y}$ such that
\[
\theta(\bold{x}^{\ast}) = \sup_{\bold{x}\in\bold{Y}} \theta(\bold{x}) = \sup_{\bold{x}\in\bold{X}} \theta(\bold{x}).
\]
Since $\bold{x}^{\ast}\in \bold{X}$, then $\bold{x}^{\ast}$ is the desired solution. The proof is finished.

\section*{Appendix 4: Proof of \Cref{thm:tau1isoptimal}}

    Fix any $\tau_1,\tau_2\in [0,1]$, with $\tau_1<\tau_2$, and let $\bold{x} = (p_1,J_1,\tau_1,e_1,I_1,Q_1,L_1)$ be a feasible point of problem $(\mathcal{P})$.\\

    Suppose first that $Q_1>0$, and therefore $L_1>0$. Moreover, since $\tau_1<1$, we have $e_1>0$.
    For every $e\in\mathbb R_+$, we can define the mapping
    \begin{align*}
    \ell_e: \mathbb R_+ &\to [\ell(e,0),+\infty)\\
    J&\mapsto \ell_e(J) = \ell(e,J).
    \end{align*}
    Note that for each $e\in \mathbb R_+$, $\ell_e$ is strictly increasing and therefore, invertible. Now, let us define
    \[
    e_{\max} = \max\{ e\ :\ \ell(e,0)\leq L_1\} \in (0,\infty).
    \]
    Clearly, $e_1\in [0,e_{\max}]$. Then, we can define the function $\phi: [0,e_{\max}]\to\mathbb R_+$ given by $\phi(e) = \ell_e^{-1}(L_1)$. Fix $e\in (0,e_{\max})$. Recalling that $\frac{\partial \ell}{\partial J}>0$ and applying the well-known Implicit Function Theorem at point $(e,J) = (e,\phi(e))$, we get that there exists a continuously differentiable function $\varphi:(e-\delta,e+\delta)\to \mathbb R$, such that
    \begin{itemize}
        \item For every $e^{\prime}\in (e-\delta,e+\delta)$, $\ell(e,\varphi(e)) = L_1$.
        \item For every $(e^{\prime},J^{\prime})$ near $(e,\phi(e))$ such that $\ell(e^{\prime},J^{\prime}) = L_1$, one has that $J^{\prime} = \varphi(e^{\prime})$.
        \item $\varphi$ is continuously differentiable on $(e-\delta,e+\delta)$, with
        \[
        \varphi'(e^{\prime}) = - \left[ \frac{\partial \ell}{\partial J}(e^{\prime},\varphi(e^{\prime}))\right]^{-1} \frac{\partial \ell}{\partial e}(e^{\prime},\varphi(e^{\prime})).
        \]
    \end{itemize}
    Noting that, by construction, for every $e^{\prime}\in (e-\delta,e+\delta)$ the only possible value for $\varphi(e^{\prime})$ must be $\ell_{e^{\prime}}^{-1}(L_1)$, we conclude that $\varphi$ coincides with $\phi$ on $(e-\delta,e+\delta)$.  Since $e$ is arbitrary, we conclude that $\phi$ is continuously differentiable on $(0,e_{\max})$ with
    \[
        \phi'(e) = - \left[ \frac{\partial \ell}{\partial J}(e,\phi(e))\right]^{-1} \frac{\partial \ell}{\partial e}(e,\phi(e)).
    \]
    Moreover, the result can be extended to the closed interval $[0,e_{\max}]$ using the continuity and strict positivity of the partial derivatives of $\ell$.

    Now, we claim that $v(\tau_2) \geq \theta(\bold{x})$. To prove the claim, define $e_2 = (1-\tau_2)p_1\frac{Q_1}{L_1}$. Clearly, $0\leq e_2<e_1\leq e_{\max}$, and so, we can define $J_2 = \phi(e_2)$. It is not hard to verify that $\bold{x}^{\prime} = (p_1,J_2,\tau_2,e_2,I_1,Q_1,L_1)$ is also a feasible point for problem $(\mathcal{P})$. We now study the claim for cases:
    \begin{enumerate}
    \item \textbf{First case, $\theta = W$:} In this case, we can write
    \[
    v(\tau_2) \geq W(\bold{x}^{\prime}) = S(p_1,T(I_1)) + p_1Q_1 - C(L_1) = W(\bold{x}).
    \]
    Observe that if $Q_1 = 0$, then $L_1 = 0$ and $T(I_1) = \infty$, and so, in this case, $J_1 = 0$. Then, the point $(p_1,J_1,\tau_2,e_1,I_1,Q_1,L_1) = (p_1,0,\tau_2,0,0,0,0)$ is also feasible. Thus, the inequality $v(\tau_2) \geq W(\bold{x})$ also holds.

    \item\textbf{Second case, $\theta = \Pi$:} In this case, let us denote $e(\tau) = (1-\tau)p_1Q_1/L_1$ and $J(\tau) = \phi(e(\tau))$. Then, we can set $\bold{y}(\tau) = (p_1, J(\tau),\tau, e(\tau),I_1,L_1,Q_1)$. With this notation, we can write
    \begin{align*}
        \frac{\partial \Pi}{\partial \tau}(\bold{y}(\tau)) &=\frac{\partial}{\partial \tau}\left( \tau p_1Q_1 - J(\tau)L_1\right)\\
        &= p_1Q_1 - L_1\frac{\partial \phi}{\partial e}(e(\tau))\frac{\partial e}{\partial \tau}\\
        &=p_1Q_1 + p_1Q_1\frac{\partial \phi}{\partial e}(e(\tau))\\
        &= p_1Q_1 \left(1 - \left[ \frac{\partial \ell}{\partial J}(e(\tau),J(\tau))\right]^{-1} \frac{\partial \ell}{\partial e}(e(\tau),J(\tau))\right) \geq 0,
    \end{align*}
    where the last inequality follows from the hypothesis that
    \[
    \frac{\partial \ell}{\partial J}(e(\tau),J(\tau)) \geq \frac{\partial \ell}{\partial e}(e(\tau),J(\tau)).
    \]
    Since this holds for every $\tau \in [\tau_1,1]$, we conclude that 
    \[
    v(\tau_2) \geq \Pi(\bold{y}(\tau_2)) \geq \Pi(\bold{y}(\tau_1)) = \Pi(\bold{x}).
    \]
    Again, if $Q_1 = 0$, then $L_1 = 0$ and $T(I_1)= \infty$, and so, in this case, $J_1 = 0$. Similarly, the point $(p_1,J_1,\tau_2,e_1,I_1,Q_1,L_1) = (p_1,0,\tau_2,0,0,0,0)$ is also feasible. Thus, the inequality $v(\tau_2) \geq \Pi(\bold{x})$ also holds. 
    \end{enumerate}
    The claim is then proven. Finally, taking supremum over all feasible points $\bold{x}=(p,J,\tau,e,I,L,Q)$ with $\tau = \tau_1$, we conclude that $v(\tau_2)\geq v(\tau_1)$.

\section*{Appendix 5: Proof of \Cref{Thm:JisZeroForLargeA}}
 Take $\epsilon>0$. We need to prove that there is an $A$ large enough so that $J^*(A)<\epsilon$. First, we note that if $J=0$, then the profit of the company is just $\Pi_0=\tau p D(p,T)$ for the resulting pickup time $T$. Crucially, all these terms are positive, and therefore $\Pi_0 \geq 0$. Thus, it suffices to show that for a large enough $A$, having $J \geq \epsilon$ would yield negative profit.

 We have assumed that $p\cdot D(p,T) \xrightarrow[]{p\rightarrow \infty}0$. Therefore, for a given $T$, the function $q_1(p)=p\cdot D(p,T)$ is bounded and achieves its maximum. Let us denote $q_2(T)=\max_p p\cdot D(p,T)$. This function is decreasing, so it achieves its maximum at $T=0$. Denoting by $q^*=q_2(0)$, we know that $\Pi(\bold{x}) \leq \tau q^* - J\cdot L$. 
 
 Let us now explore what happens if $J \geq \epsilon$. Then $J \cdot L \geq \epsilon \cdot L = \epsilon \cdot A \cdot \rho(e,J) \geq \epsilon \cdot A \cdot \rho(0,J) \geq \epsilon \cdot A \cdot \rho(0,\epsilon)$. This means that

 \begin{equation*}
     \Pi(\bold{x}) \leq \tau q^* - \epsilon A \rho (0,\epsilon)
 \end{equation*}

Clearly, if $A$ is too large, this profit becomes negative, and in particular worse than $\Pi_0$. This shows that $J \geq \epsilon$ could not be optimal, which completes the proof.

\section*{Appendix 6: Numerical value of the parameters}

\begin{table}[H]
\centering
\begin{tabular}{|c|c|c|}
\hline
\textbf{Parameters} & \textbf{Numerical value} & \textbf{Unit} \\
\hline
$\kappa$ & 1.768 & Dimensionless \\
\hline
$\beta_p$ & -0.669 & 1/hour \\
\hline
$\beta_T$ & -1.134 & 1/hour \\
\hline
$K_T$ & 0.127 & Hour \\
\hline
$\alpha_T$ & -0.515 & Dimensionless \\
\hline
$\varepsilon_L$ & 1.2 & Dimensionless \\
\hline
$t$ & 0.25 & Hour \\
\hline
\end{tabular}
\caption{Numerical value of the parameters that do not depend on $h$.}
\label{tab:Parameters and Numerical values}
\end{table}

\begin{table}[H]
\centering
\begin{tabular}{|c|c|c|}
\hline
\textbf{Hour} & \textbf{$\lambda_h$} & \textbf{$A_h$} \\
\hline
1 & 30 & 13 \\
2 & 15 & 11 \\
3 & 10 & 12 \\
4 & 5 & 4.5 \\
5 & 15 & 4 \\
6 & 18 & 6 \\
7 & 39 & 11 \\
8 & 70 & 17.5 \\
9 & 120 & 22 \\
10 & 100 & 32.5 \\
11 & 80 & 28.5 \\
12 & 77 & 28 \\
13 & 73 & 25 \\
14 & 77 & 23 \\
15 & 79 & 23.3 \\
16 & 85 & 24 \\
17 & 100 & 25 \\
18 & 145 & 29 \\
19 & 163 & 45 \\
20 & 150 & 50 \\
21 & 130 & 43 \\
22 & 120 & 32 \\
23 & 110 & 29 \\
24 & 70 & 28.5 \\
\hline
\end{tabular}
\caption{Numerical value of the parameters that depend on $h$.}
\label{tab:Variables that depend on h}
\end{table}

\end{document}